\documentclass[11pt,reqno,a4paper]{amsart}

\usepackage{array}
\usepackage{enumitem}
\usepackage{mathtools}
\usepackage{pgfplots}
\usepackage{soul}
\usepackage{bm} 

\usepackage{amsmath,amssymb,amsfonts}
\usepackage{tikz}
\usepackage{hyperref}
\usepackage{appendix}
\usepackage{scalerel}
\usepackage{graphicx}
\usepackage{subcaption}
\usepackage{amsthm}
\usepackage{multirow}
\usepackage{pdflscape}
\usepackage{afterpage}
\usepackage{capt-of} 
\usepackage{lipsum}
\usepackage{booktabs}
\usepackage{siunitx}
\usepackage{esvect}

\usepackage{cprotect}

\tikzset{decorated arrows/.style={
    postaction={
        decorate,
        decoration={
            markings,
            mark=between positions 0 and 1 step 15mm with {\arrow[black]{stealth};}
            }
        },
    }
}
 
\tikzset{decorated arrows2/.style={
    postaction={
        decorate,
        decoration={
            markings,
            mark=at position 15mm with {\arrow[black]{stealth};}
            }
        },
    }
}

\usepgfplotslibrary{colormaps}
\usetikzlibrary{arrows}

\makeatletter
\def\set@curr@file#1{%
  \begingroup
    \escapechar\m@ne
    \xdef\@curr@file{\expandafter\string\csname #1\endcsname}%
  \endgroup
}
\def\quote@name#1{"\quote@@name#1\@gobble""}
\def\quote@@name#1"{#1\quote@@name}
\def\unquote@name#1{\quote@@name#1\@gobble"}
\makeatother
\usepackage{graphics}

\newtheorem{lemma}{Lemma}

\newtheorem{maintheorem}{Theorem}

\newtheorem{defn}{Definition}[]

\newtheorem{rem}{Remark}

\theoremstyle{plain}

\newcommand{\NN}{{\mathbb{N}}}

\newcommand{\RR}{\mathbb{R}}

\newcommand{\dpt}{\displaystyle}

\newcommand{\RN}[1]{%
  \textup{\uppercase\expandafter{\romannumeral#1}}%
}
      
\author[J. P. S. M. de Carvalho and A. A. Rodrigues]{Jo\~ao P. S. Maur\'icio de Carvalho$^{*\MakeLowercase{a},1,2}$ and Alexandre A. Rodrigues$^{\MakeLowercase{b},1,2,3}$ \\
\\
$^\MakeLowercase{a}$\MakeLowercase{jocarvalho@fc.up.pt} \\ \medskip {\bf ORCID:} 0000-0001-7709-1631 \\
$^\MakeLowercase{b}$\MakeLowercase{alexandre.rodrigues@fc.up.pt} \\ {\bf ORCID:} 0000-0001-8182-9889 \\
\\ 
$^{1}$Faculty of Sciences, University of Porto, \\ \medskip Rua do Campo Alegre s/n, Porto 4169-007, Portugal \\ 
$^{2}$Centre for Mathematics, University of Porto, \\ \medskip Rua do Campo Alegre s/n, Porto 4169-007, Portugal \\
$^3$Lisbon School of Economics \& Management, \\Rua do Quelhas 6, \\ Lisboa 1200-781, Portugal 
}

\begin{document}

\thanks{JPSMC was supported by CMUP, Portugal (UIDP/MAT/00144/2020), which is funded by Funda\c{c}\~ao para a Ci\^encia e a Tecnologia (FCT). AR was partially supported by CMUP, Portugal (UIBD/MAT/00144/2020), which is funded by FCT with national and European structural funds through the programs FEDER, under the partnership agreement PT2020. \\ $^*$Corresponding author.}

\title[SIR model with vaccination: bifurcation analysis]
{SIR model with vaccination: \\bifurcation analysis}

\date{\today}  

\begin{abstract}
There are few adapted SIR models in the literature that combine vaccination and logistic growth.
In this article, we study bifurcations  of a SIR model where the class of \emph{Susceptible} individuals grows logistically and has been subject to constant vaccination. We explicitly prove that the endemic equilibrium is a \emph{codimension two singularity} in the parameter space $(\mathcal{R}_0, p)$, where $\mathcal{R}_0$ is the \emph{basic reproduction number} and $p$ is the proportion of \emph{Susceptible} individuals successfully vaccinated at birth. 

We exhibit explicitly the \emph{Hopf}, \emph{transcritical}, \emph{Belyakov}, \emph{heteroclinic} and \emph{saddle-node bifurcation} curves unfolding the \emph{singularity}. 
The two parameters  $(\mathcal{R}_0, p)$ are written in a useful way to evaluate the proportion of vaccinated individuals necessary to eliminate the disease and to conclude how  the vaccination may affect the outcome of the epidemic. We also exhibit  the region in the parameter space where the disease 
persists and we illustrate our main result with numerical simulations, emphasizing the role of the parameters. 
\end{abstract}

\maketitle



\section{Introduction}

Mathematical models applied to epidemiology play an important role to understand the dynamics of infectious diseases. Understanding how a  disease evolves in a population can be  beneficial, not only to predict epidemic outbreaks, but also to improve approaches against its spread~\cite{Brauer2019,Hethcote2000,Bonyah2020,Rajagopal2020,Cobey2020}. One of the most widely used models to study population interactions is the SIR model based on the division of the population in three classes of individuals: \emph{Susceptible} (S), \emph{Infectious} (I) and \emph{Recovered} (R)~\cite{Kermack1932,Dietz1976}.

When a disease enters a population, the medical community aims (as quick as possible) to find ways to stop its evolution. 
There are studies in the literature that investigate the behaviour of populations in the presence of infectious diseases by using a wide range of techniques~\cite{MilnerPugliese1999,CarvalhoRodrigues2022,Barrientos2017,CarvalhoPinto2021,Onofrio2022}. 
One of the most effective ways to prevent the disease's progression is via a \emph{vaccination} policy~\cite{Plotkin2005,Plotkin2011, Makinde2007,SahaGhosh2022,Ghosh2019}. As far as we know, there are few articles in the literature that combine logistic growth in the \emph{Susceptible} population and the effect of vaccination.
\medbreak
In mathematical modelling, one may identify three vaccination strategies:
\begin{enumerate}
\item[(i)] \emph{Constant vaccination}~\cite{Makinde2007,Shulgin1998,Elazzouzi2019} -- it consists of vaccinating a prescribed ratio of newborns;
\item[(ii)] \emph{Pulse vaccination}~\cite{Shulgin1998,Stone2000} -- it implies vaccinating a percentage of the susceptible population periodically;
\item[(iii)] \emph{Mixed vaccination}~\cite{Shulgin1998} -- combination of constant and pulse vaccination.\\
\end{enumerate}

Finding the most appropriate strategy is a challenge because we need to combine the effect of several factors, namely the efficiency of the vaccination and its cost to the public health policies. 
The application of \emph{bifurcation analysis} to epidemiology  may  give clues about the evolution of a given disease in the presence of several factors namely vaccination, seasonality, sub-optimal immunity and nonlinear incidence~\cite{Algaba2022,Rodrigues2021,Yagasaki2002,Duarte2019}.  In the present article, we are interested in a modified SIR model exhibiting a \emph{Double-zero (DZ) singularity}  and the dynamics it may unfold~\cite{Kuznetsov2004}. Sensitivity analysis may be particularly useful in this context.

\subsection*{State of art}
There are several studies involving \emph{singularities} in epidemic models\footnote{There is an abundance of references in the literature. We choose to mention only a few for clarity and the choice is uniquely based on our  preferences. The reader interested in more details and examples may use the references within those we mention.}.

In Jin \emph{et al.}~\cite{Jin2007}, the authors studied global dynamics of a SIRS model with nonlinear incidence rate. They established a threshold for a disease to be extinct or endemic, analyzed the existence and  stability of equilibria and verified the existence of bistable states. Using normal forms and the Dulac criterion, they investigated  backward bifurcation and obtained all curves associated to the {\it Bogdanov-Takens bifurcation}.

Zhang and Qiao~\cite{Zhang2023}  analyzed bifurcations in a SIR model including  bilinear incidence rate,  vaccination and hospital resources (number of available beds).
They exhibited conditions to ensure the existence of a codimension three singularity for the system. 
They concluded that a high value of the vaccination parameter allows the disappearance of the disease in the population.  See also~\cite{ShanZhu2014} where the authors considered a SIR model with a standard incidence rate and a nonlinear recovery rate, formulated to consider the impact of available resources of the public health system (essentially the number of hospital beds). 

Alexander and Moghadas~\cite{AlexanderMoghadas2005} studied a SIRS epidemic model with a generalized nonlinear incidence as a function of the number of infected individuals. It is assumed that the natural immunity acquired by the infection is not permanent but wanes with time.  
 Normal forms have been  derived for the different types of bifurcation that the model undergoes.  The {\it Bogdanov-Takens} normal form has been used to formulate the local bifurcation curves for a family of {\it homoclinic orbits} arising when a {\it Hopf} and a {\it saddle-node bifurcation} merge. They provided conditions for the occurrence of {\it Hopf bifurcations} in terms of two  parameters: the {\it basic reproductive number} and the rate of loss of immunity acquired by the infection.

In 2022, Pan \emph{et al.}~\cite{Pan2022} proposed a SIRS model undergoing a degenerate codimension three bifurcation inducing intermittency into the dynamics. The authors provided sufficient conditions to guarantee the global asymptotic stability of the unique endemic equilibrium. Vaccination and its boosted version have not been taken into account.

In 2018, Li and Teng~\cite{LiTeng2018} presented a SIRS model with generalized non-monotone incidence rate and qualitatively proved the existence of \emph{Bogdanov-Takens bifurcations}.  They  located regions where the disease either persists or disappears. This phenomenon indicates that the initial conditions of an epidemic may determine the final states of an epidemic to go extinct or not. We also address the reader to~\cite{Misra2022,Lu2019} for other variations of the SIR/SIRS models. The reference~\cite{Lu2019} includes numerical simulations and data-fitting of the influenza data in China.

\subsection*{Novelty} Our work contributes to the mathematical understanding of the modified SIR model, where the class of \emph{Susceptible} individuals (subject to a constant vaccination) grows logistically.
We find a DZ singularity in the bifurcation space $(\mathcal{R}_0, p)$, where $\mathcal{R}_0$ is the \emph{basic reproduction number} and $p$ is the proportion of \emph{Susceptible} individuals successfully vaccinated at birth. Writing the parameters (of the model) as function of $\mathcal{R}_0$ and $p$ is our first breakthrough. 

We exhibit explicit expressions for  the \emph{saddle-node}, \emph{transcritical}, \emph{Hopf} and \emph{heteroclinic bifurcation} curves associated to the DZ bifurcation. 
These unfolding curves have the same qualitative properties to the truncated amplitude system associated to the \emph{Hopf-zero} normal form (8.81) of~\cite{Kuznetsov2004}. The \emph{heteroclinic cycle} is associated to two disease-free equilibria and is asymptotically unstable (repelling). 

For the sake of completeness,  in Section \ref{s:prel}, we describe some terminology that are going to be useful throughout this article.  

\section{Preliminaries}
\label{s:prel}
In this section, we introduce some terminology for vector fields acting on a $\RR^n$, $n\in \NN$, that will  be used in the remaining sections. Let $f$ be a smooth vector field on $\RR^n$ with flow given by the unique solution $x(t)=\varphi(t, x)\in \RR^n$ of the two-parameter family
\begin{equation}\label{general}
\dot{x}=f_{(\eta_1, \eta_2)} (x), \qquad x(0)=x_0,
\end{equation}
where $(\eta_1, \eta_2)\in \RR^2$.  \\

\begin{defn}
We say that $\mathcal{K} \subset (\RR^+_0)^2$ is a \emph{positively flow-invariant set} for \eqref{general} if for all $x \in \mathcal{K}$ the trajectory of $\varphi(t, x)$ is  contained in $\mathcal{K}$ for $t \geq 0$.\\
 \end{defn}
\begin{defn}
 Given two hyperbolic equilibria $A$ and $B$ of \eqref{general}, a \emph{heteroclinic connection} from $A$ to $B$, is a solution of \eqref{general} 
contained in $W^u(A)\cap W^s(B)$, the intersection of the unstable manifold of $A$ and the stable manifold of $B$.\\
 \end{defn}

For a solution of (\ref{general}) passing through $x\in \RR^n$, the set of its accumulation points, as $t$ goes to $+\infty$, is the \emph{$\omega$-limit set} of $x$. More formally, if $\overline{A}$ is the topological closure of $A\subset \RR^n$, then: 
 \begin{defn} If $x\in \RR^n$, the $\omega$-limit of $x$ is:
$$
\omega(x)=\bigcap_{T=0}^{+\infty} \overline{\left(\bigcup_{t>T}\varphi(t, x)\right)}.
$$ 
 \end{defn}
It is well known that $\omega(x)$  is closed and flow-invariant, and if the $\varphi$--trajectory of $x$ is contained in a compact set, then 
$\omega(x)$ is non-empty. If $E$ is an invariant set of \eqref{general}, we say that $E$ is a \emph{global attractor} if $\omega(x) \subset E$, for Lebesgue almost all points $x$ in   $\mathbb{R}^n$.

\medskip

The \emph{center manifold} of a non-hyperbolic equilibrium is the set of solutions whose behaviour around the equilibrium point is not controlled neither by the attraction of the stable manifold nor by the repulsion of the unstable manifold.  If the linearized part of $Df_{(\eta_1, \eta_2)}$ (at a given equilibrium) has an eigenvalue with zero  real part, the center manifold plays an important goal and it is the right set where bifurcations occur. 
 
 Throughout this article, we study the DZ \emph{singularity} of codimension two   for a family of differential equations  corresponding to the case $s=1$, $\theta<0$ in Equation (8.81) of~\cite{Kuznetsov2004}.  The unfolding of this \emph{singularity} involves lines of \emph{saddle-node}, \emph{Belyakov transition}, \emph{Hopf} and \emph{homo/heteroclinic bifurcations}. We suggest the reading of~\cite{Kuznetsov2004} for a complete understanding of these bifurcations as well as the sufficient conditions that prompt their existence.

\section{The model} \label{model_model}

Inspired by the classical SIR model, we are going to divide the individuals of a given (human) population into three classes of individuals~\cite{Dietz1976,LiTeng2017}:

\medskip
 
\begin{itemize}
\item \emph{Susceptible (S)}: proportion of healthy individuals who are susceptible to the disease;
\medskip
\item  \emph{Infectious (I)}: proportion of infected individuals who can transmit the disease to susceptible individuals; 
\medskip
\item \emph{Recovered (R)}: proportion of individuals who recovered naturally from the disease or through immunity conferred by the vaccine. This comprises individuals who have definitive immunity and can not transmit the disease. \\
\end{itemize}

We assume that \emph{Susceptible} individuals have never been in contact with the disease, but may become infected when they are in contact with the population of the \emph{Infectious}, and then become part of this class. Whereas in the class of \emph{Infectious}, these individuals can recover naturally and become part of the class of \emph{Recovered} ones. The \emph{Susceptible} individuals may also have been successfully vaccinated at birth, thus becoming immune to the disease~\cite{Shulgin1998,StoneShulgin2000,LuChiChen2002}. Inspired by~\cite{Dietz1976,Shulgin1998,ZhangChen1999}, the nonlinear system of ODE in variables $S$, $I$ and $R$ (depending on time $t$), is given by:

\begin{equation}
\label{modeloSIR}
\begin{array}{lcl}
\dot{X} = \mathcal{F} (X) \quad \Leftrightarrow \quad
\begin{cases}
&\dot{S} =  S(A-S) - \beta I S - pm \\
\\
&\dot{I} =  \beta IS - (\mu+d) I - gI \\
\\	
&\dot{R} =  p m + gI - \mu R,
\end{cases}
\end{array}
\end{equation}

\smallskip

\noindent where 

\smallskip

$$
\begin{array}{lcl}
X(t) &=& \left( S(t), I(t), R(t) \right) \in (\RR^+)^3, \\
\\
\dot{X} &=& (\dot{S}, \dot{I}, \dot{R}) \,\,\, = \,\,\, \dpt \left(\frac{\mathrm{d}S}{\mathrm{d}t},\frac{\mathrm{d}I}{\mathrm{d}t},\frac{\mathrm{d}R}{\mathrm{d}t}\right). \\
\end{array}
$$
\bigbreak

\begin{rem}
\label{extension}
When $S=0$, we may  extend non-smoothly $\mathcal{F}$ to the vector field $$(0, - (\mu+d) I - gI, pm + gI - \mu R).$$ 
\end{rem}

\medskip

The parameters of \eqref{modeloSIR} may be interpreted as follows:

\medskip

\begin{description}
\item[$A$] carrying capacity of susceptible individuals when $\beta=0$ (in the absence of disease) and $p = 0$ (in the absence of vaccination);
\smallskip
\item[$\beta$] transmission rate of the disease;
\smallskip
\item[$m$] birth rate; 
\smallskip
\item[$p$] proportion of susceptible individuals successfully vaccinated at birth, for $p \in [0,1]$;
\smallskip
\item[$\mu$] natural death rate of infected and recovered individuals;
\smallskip
\item[$d$] death rate of infected individuals due to the disease;
\smallskip
\item[$g$]  natural recovery rate. \\
\end{description}

Figure \ref{boxes} illustrates the interaction between the classes of \emph{Susceptible}, \emph{Infectious} and \emph{Recovered} individuals in model \eqref{modeloSIR}.
The \emph{basic reproduction number}, denoted by $\mathcal{R}_0$, may be seen as the number of secondary infections caused by a single infected person in a susceptible population~\cite{Jones2007}. For model \eqref{modeloSIR} with $p=0$, $\mathcal{R}_0$ may be explicitly computed as~\cite{CarvalhoRodrigues2022,ParkBolker2020,Li2011}:
\begin{equation}
\label{R0}
\begin{array}{lcl}
\mathcal{R}_0 =\dpt \lim_{T\rightarrow +\infty} \frac{1}{T} \int_0^T \dfrac{A \beta}{\mu+d+ g} \, \mathrm{d}t = \dfrac{A \beta}{\mu+d+ g}  > 0.
\end{array}
\end{equation}

\subsection{Hypotheses and motivation}

With respect to system \eqref{modeloSIR}, we also assume the following conditions:
\smallskip

\begin{itemize}
\item[{\bf (C1)}] All parameters are positive;
\medskip
\item[{\bf (C2)}] For all $t \in \mathbb{R}^+_0, S(t) \leq A$;
\medskip
\item[{\bf (C3)}] Vaccination is considered only when $\mathcal{R}_0 > 1$, \emph{i.e.} $p>0$ if and only if $\mathcal{R}_0 > 1$ (in other words, for $\mathcal{R}_0 < 1$ no preventive measures involving vaccination will be considered). 
\end{itemize}

\medskip

\noindent The phase space associated to \eqref{modeloSIR} is a subset of $(\mathbb{R}_0^+)^3$ induced with the usual topology, and the set of parameters is: 
\begin{eqnarray*}
\label{parameters_0}
\Omega = \left\{ \omega = (A,\beta,m,\mu,d,g) \in (\mathbb{R}^+)^6 \right\} \quad \text{and} \quad p \in \mathbb{R}_0^+.
\end{eqnarray*}

\usetikzlibrary{arrows,positioning}
\begin{center}
\begin{figure}[ht!]
\begin{tikzpicture}
[
auto,
>=latex',
every node/.append style={align=center},
int/.style={draw, minimum size=1.75cm}
]

    \node [fill=lightgray,int] (S)             {\Huge $S$};
    \node [fill=lightgray,int, right=5cm] (I) {\Huge $I$};
    \node [fill=lightgray,int, right=11cm] (R) {\Huge $R$};
    
    \node [below=of S] (s) {} ;
    \node [below=of I] (i) {} ;
    \node [below=of R] (r) {};
    
    \node [above=of S] (os) {};
    \node [above=of I] (oi) {};
    \node [above=of R] (or) {};
    
    \coordinate[right=of I] (out);
    \path[->, auto=false,line width=0.35mm]
    			(S) edge node {$\beta I S$ \\[.6em]} (I)
                          (I) edge node {$gI$       \\[.6em] } (R) ;

    \path[->, auto=false,line width=0.35mm]
			(S) edge  [out=-120, in=-60] node[below] {$S(A-S)$ \\ [0.2em] \emph{Logistic growth}} (S)
    			(I) (5.88,-0.87cm) edge [] node[right]{$(\mu+d) I$} (5.88,-2cm) (i) 
    			(R) (11.88,-0.87cm) edge [] node[right]{$\mu R$} (11.88,-2cm) (r) ;

    \path[-, auto=false,line width=0.35mm]
    			(S) (0,0.87cm) edge [] node[right]{} (0,2cm) (os)
			(os) (0,2cm) edge [] node[above]{$p m$} (11.88,2cm) (or)
			(or) (11.88,2cm) edge [->] node[below]{} (11.88,0.87cm) (R) ;
                         
\end{tikzpicture}
\caption{\small Schematic diagram of model \eqref{modeloSIR}. Boxes represent compartments, and arrows indicate the flow between $S$, $I$ and $R$.}
\label{boxes}
\end{figure}
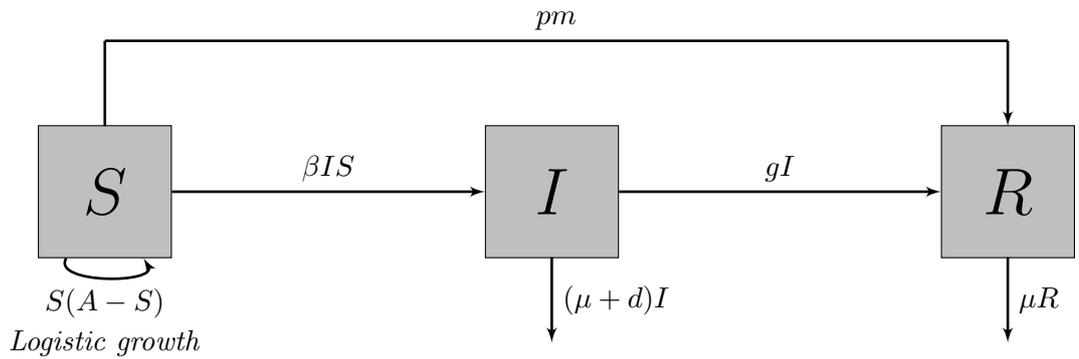
\end{center}

System \eqref{modeloSIR} has been inspired in the classical SIR model~\cite{Kermack1932} with slight modifications as we proceed to explain:

\smallskip

\begin{itemize}
\item We   consider logistic growth in the susceptible population due to the crowding and natural competition for resources~\cite{CarvalhoRodrigues2022,LiTeng2017,ZhangChen1999}, instead of linear or exponential growth;
\medskip
\item Instead of a mixed or pulse vaccination strategy~\cite[Subsection 2.1]{Shulgin1998}, we have assumed constant vaccination~\cite{Makinde2007,Shulgin1998,Stone2000}.
\end{itemize}

%

\subsection{Two-dimensional system} 

The first two equations of (\ref{modeloSIR}), $\dot{S}$ and $\dot{I}$, are independent of  ${R}$. Therefore,  we may consider the system:

\smallskip

\begin{equation}
\label{modelo2}
\begin{array}{lcl}
   
\begin{cases}
&\dot{S} =  S(A-S) - \beta I S - pm \\
\\
&\dot{I} =  \beta IS - \left( \mu+d \right)I - g I,
\end{cases}
\end{array}
\end{equation}

\medskip

\noindent with $x=(S,I)\in (\RR^+)^2$. Because of Remark \ref{extension},  we may extend  \eqref{modelo2} to the line $S=0$:
\begin{equation}
\label{modelo2 e S=0}
 \begin{array}{lcl}
 \begin{cases}
&\dot{S} =  0 \\
\\
&\dot{I} =    - \left( \mu+d \right)I - g I.
\end{cases}
\end{array}
\end{equation} 

The vector field associated to \eqref{modelo2} and \eqref{modelo2 e S=0} will be denoted by $f$ and its flow is $$\varphi(t, (S_0, I_0)), \quad t \in \mathbb{R}^+_0,\quad (S_0, I_0) \in (\mathbb{R}^+_0)^2.$$ This model  will be the object of study of the present article; in order to  shorten the notation,  we denote by $\sigma$ the sum $\mu + d$.

\begin{lemma} \label{positivity}

The region defined by:
$$
\mathcal{M} = \left\{ (S,I) \in  (\mathbb{R}_0^+)^2: \quad  0 \leq S \leq A, \quad 0 \leq S+I \leq \dfrac{A (\sigma + g + A)}{\sigma + g}, \quad S, I\geq 0 \right\},
$$
is positively flow-invariant for  \eqref{modelo2} and \eqref{modelo2 e S=0}.

\end{lemma}

\begin{proof}
 It is easy to check that $(\mathbb{R}_0^+)^2$ is flow invariant (note that \eqref{modelo2 e S=0} leaves invariant the vertical lines $S=0$ and $I=0$).

Now, we show that if $(S_0, I_0)\in \mathcal{M}$, then $\varphi_0(t, (S_0, I_0))$, $t \in \mathbb{R}_0^+$, is contained in $ \mathcal{M}$.
Let us define  $\phi(t) = S(t) + I(t)$ associated to the trajectory $\varphi(t, (S_0, I_0))$.  Omitting the dependence of the variables on $t$ (when there is no risk of ambiguity), one knows that:

\begin{equation}
\label{proof1}
\begin{array}{lcl}
\dot{\phi} & = & \dot{S} + \dot{I} \\
\\
& = & S(A-S) - \beta I S - pm + \beta I S - (\sigma + g) I
\nonumber \\
\\
& = & S(A-S) - pm - (\sigma + g)I ,\\
\end{array}
\end{equation}

\bigbreak 

from where we deduce that:

\bigbreak 

\begin{equation}
\label{proof2}
\begin{array}{lcl}
\dot{\phi} + (\sigma + g) \phi & = & S(A-S) - pm - (\sigma + g)I + (\sigma + g)S + (\sigma + g)I
\nonumber \\
\\
& = & S(A-S) + (\sigma + g)S  \\ \\
& \leq  &   (\sigma + g + A) S.\\
\end{array}
\end{equation}
\bigbreak

If $\beta=p=0$, then the first component of \eqref{modelo2} would be the logistic growth and thus its solution is limited by $A$ (see {\bf (C2)}), a property which remains for $\beta,p > 0$. In particular, we may conclude that 
\begin{equation*}
\label{proof3}
\begin{array}{lcl}
\dot{\phi} + (\sigma + g) \phi \leq (\sigma + g + A) A.
\end{array}
\end{equation*}
The classical differential version of the Gronwall's Lemma\footnote{If $a,b \in \RR$ and $u: \RR_0^+\rightarrow \RR_0^+$ is a $C^1$ map  such that  $u'\leq au+b$, then $u(t)\leq u(0) e^{at}+\frac{b}{a}(e^{at}-1)$.} says that for all $t \in \RR^+_0$, we have:
$$
\phi(t)\leq \phi(0) e^{-(\sigma+g) t} - \frac{(\sigma + g + A) A}{(\sigma+g)}\left(e^{-(\sigma+g) t}-1\right).
$$
Taking the limit when $t\rightarrow +\infty$, we get:
\begin{eqnarray*}
0\leq \lim_{t\rightarrow +\infty} \phi(t)&\leq& \lim_{t\rightarrow +\infty}  \left[\phi(0) e^{-(\sigma+g) t} - \frac{(\sigma + g + A) A}{(\sigma+g)}\left(e^{-(\sigma+g) t}-1\right) \right]\\
&=&   \frac{(\sigma + g + A) A}{(\sigma+g)}\\
\end{eqnarray*}
 Since $\dpt \lim_{t\rightarrow +\infty} \phi(t)= \dpt \lim_{t\rightarrow +\infty}\left( S(t) + I(t)\right)$, the result follows.

 \end{proof}

\section{Main result and consequences}
We state the main results of the article, as well as its structure. We also discuss some consequences.

\subsection{Main result}
System \eqref{modelo2} may have three formal equilibria\footnote{The term ``\emph{formal equilibria}'' means that they are equilibria of the system regardless it makes sense or not in the context of the epidemic problem.}: two disease-free equilibria and one endemic equilibrium, when they exist.
The disease-free equilibria of system \eqref{modelo2} are:

\begin{eqnarray*}
\label{S0_ref}
E_0^p =(S_0^p, I_0^p)= \left(\dfrac{A-\sqrt{A^2 - 4pm}}{2},0 \right)
\end{eqnarray*}
and 
\begin{eqnarray*}
\label{S1_ref}
E_1^p = (S_1^p,I_1^p)= \left(\dfrac{A+\sqrt{A^2 - 4pm}}{2},0 \right),
\end{eqnarray*}


\noindent  where $p \leq \frac{A^2}{4m} \leq 1$, which implies $ A \leq 2\sqrt{m}$. The endemic formal equilibrium of \eqref{modelo2} is:

\begin{eqnarray*}
E_2^p =\left(S_2^p,I_2^p\right)= \left(\dfrac{\sigma + g}{\beta},\dfrac{-pm\beta^2 + A\left(\sigma+g\right)\beta-\left(\sigma+g\right)^2}{\beta^2 \left(\sigma+g\right)}\right),
\end{eqnarray*}

\noindent where $\dpt -pm\beta^2 + A\left(\sigma+g\right)\beta-\left(\sigma+g\right)^2 > 0 \, \overset{\eqref{R0}}\Leftrightarrow \,\mathcal{R}_0 > 1 + \frac{pm\beta^2}{\left( \sigma + g \right)}$.\\

We are able to study the map $f$ as a two-parameter family depending on the \emph{basic reproduction number} $\mathcal{R}_0$ and the \emph{proportion of vaccination} $p$.
 Our main result says that, in the bifurcation parameter $(\mathcal{R}_0 , p)$, $E_2^p$ is a DZ \emph{singularity} for the vector field $f \mapsto f_{(\mathcal{R}_0 , p)}$. 
\begin{maintheorem}
\label{th: mainA}
The endemic equilibrium $E_2^p$ of \eqref{modelo2} undergoes a DZ bifurcation at $(\mathcal{R}_0^{\star},p^\star) = \left(2,   \frac{A^2}{4m}\right)$. The local representations of the bifurcation curves in the space of parameters $(\mathcal{R}_0,p) \in (\mathbb{R}_0^+)^2$ are as follows:

\medskip

\begin{itemize}
\item[(i)] Saddle-node bifurcation curve:
$$
\boldsymbol{SN} = \left\{ (\mathcal{R}_0,p) \in (\mathbb{R}_0^+)^2 \, : \,  p  =  \frac{A^2}{4m} \right\}
$$

\item[(ii)] Transcritical bifurcation curve:
$$
\boldsymbol{T} = \left\{ (\mathcal{R}_0,p) \in (\mathbb{R}_0^+)^2 \, : \,  p  = \dfrac{A^2}{m} \dfrac{\left(\mathcal{R}_0-1\right)}{\mathcal{R}_0^2}  \right\}
$$

\item[(iii)] Belyakov transition curve:
$$
\boldsymbol{B_t} = \left\{ (\mathcal{R}_0,p) \in (\mathbb{R}_0^+)^2 \, : \,  p =  \dfrac{\left(-2\beta+1+2\displaystyle\sqrt{\beta\left(\mathcal{R}_0 + \beta - 2\right)}\right)A^2}{m\mathcal{R}_0^2}\right\}
$$

\item[(iv)] Heteroclinic cycle bifurcation curve:
$$
{\bf Het.} = \left\{ (\mathcal{R}_0,p) \in (\mathbb{R}_0^+)^2 \, : \,  p \approx \dfrac{4.495}{\mathcal{R}_0^{2.313}} - 0.039 \right\}
$$

\item[(v)] Hopf bifurcation curve:
$$
\boldsymbol{H} = \left\{ (\mathcal{R}_0,p) \in (\mathbb{R}_0^+)^2 \, : \,  p  = \dfrac{A^2}{m\mathcal{R}_0^2} \right\}.
$$
\end{itemize}
\end{maintheorem}

The \emph{Belyakov transition} is a curve in the bifurcation space where   the eigenvalues of $ Df_{(\mathcal{R}_0 , p)} $ at $E_2^p$  change from non-real to real or vice-versa. 

\begin{figure}[h!]
\center
\includegraphics[scale=0.60]{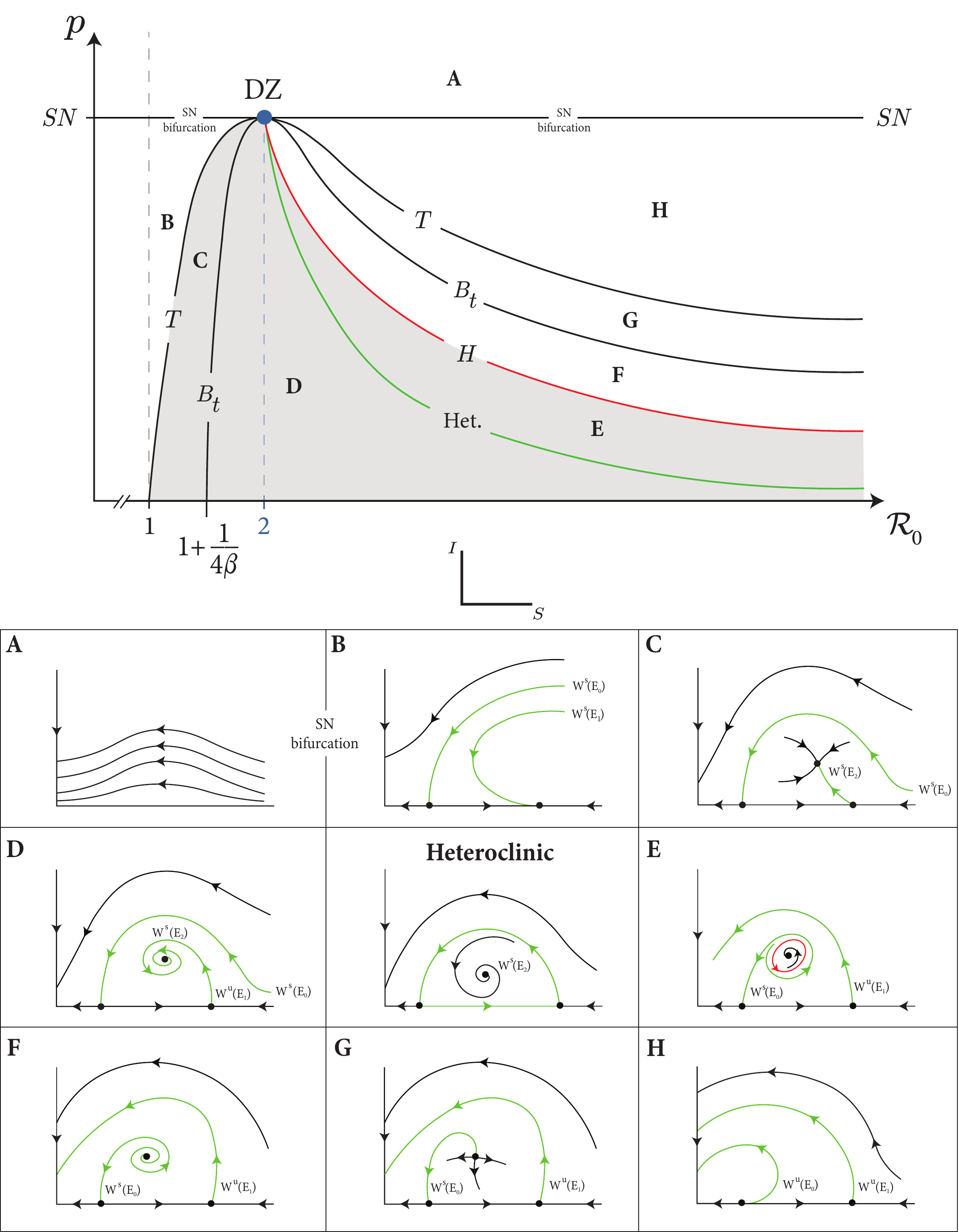}
\caption{DZ \emph{bifurcation} diagram associated to \eqref{modelo2}. In the regions {\bf C}, {\bf D} and {\bf E}, the disease  persists in the population. Compare with numerics of Figure \ref{subplots}. }
\label{ZH_esquema}
\end{figure}

\subsection{Proof of Theorem \ref{th: mainA} and the structure of the article} \label{PROVE_THEO_A}

The jacobian matrix of the vector field  associated to \eqref{modelo2} at $E_2^p$ when $(\mathcal{R}_0^{\star},p^\star) = \left(2,   \frac{A^2}{4m}\right)$ is given by: 
\begin{eqnarray}
\nonumber &&\left(\begin{array}{cc}
\dfrac{p^\star m\beta^2-\left(\sigma+g\right)^2}{\beta\left(\sigma+g\right)}  & - \left(\sigma+g\right) \\ 
\\
\dfrac{-p^\star m\beta^2 + A\left(\sigma+g\right) \beta - \left(\sigma+g\right)^2}{\beta\left(\sigma+g\right)} & 0
\end{array}\right) \\
\nonumber \\
\nonumber &=&\left(\begin{array}{cc}
\dfrac{p^\star m}{A}\mathcal{R}_0^{\star} - \dfrac{A}{\mathcal{R}_0^{\star}}  & - \left(\sigma+g\right) \\ 
\\
-\dfrac{p^\star m}{A}\mathcal{R}_0^{\star} + A - \dfrac{A}{\mathcal{R}_0^{\star}} & 0
\end{array}\right) \\
\nonumber \\
&=&\left(\begin{array}{cc}
0  & - \left(\sigma+g\right) \\ 
\\
0 & 0
\end{array}\right) . \label{JE2_DZb}
\end{eqnarray}

It is easy to verify that the matrix \eqref{JE2_DZb} is non-hyperbolic and has a double zero eigenvalue. This is why we say that, for $(\mathcal{R}_0^{\star},p^\star) = \left(2,   \frac{A^2}{4m}\right)$, the point $E_2^p$ is a \emph{singularity} of codimension 2.
The associated DZ \emph{bifurcations} exist if the vector field $f_{(\mathcal{R}_0 , p)}$, at the bifurcation point, satisfies the nondegeneracy conditions described in~\cite[pp. 316--322]{Kuznetsov2004}. 
Instead of verifying these additional conditions, we check analytically the existence of all  bifurcation curves that pass through  the point $ (\mathcal{R}_0^{\star},p^\star) =\left(2,  \frac{A^2}{4m}\right)$, as pointed out in Table \ref{notation2}.

\setlength{\tabcolsep}{12pt}
\renewcommand{\arraystretch}{1.35}

\begin{table}[h!]
\centering
\begin{tabular}{ |l|c|c|  }
 \hline
Bifurcation/Transition & Curve in Figure \ref{ZH_esquema} & Section  \\  [0.29ex]
 \hline  
 \emph{Saddle-node}  &$SN$  & \ref{s:sn}  \\ 
\emph{Transcritical} &$T$ & \ref{s:transcritical}  \\ 
\emph{Belyakov} &$B_t$ & \ref{s:Belyakov}   \\ 
\emph{Hopf} &$H$ & \ref{s:Hopf} \\ 
\emph{Heteroclinic} & Het. & \ref{s:Heteroclinic} \\ 
 \hline
\end{tabular}
\medskip
\cprotect \caption{Codimension 1 bifurcations that characterizes the DZ \emph{bifurcation} and structure of  the proof of Theorem \ref{th: mainA}.}
\label{notation2}
\end{table}

For the sake of completeness, we have added in Section  \ref{p0} a study of model \eqref{modelo2} for $p=0$. 
The curve representing the \emph{heteroclinic cycle}, denoted by $\mathcal{H}$, in the space of parameters $(\mathcal{R}_0, p)$ was obtained by interpolation. The estimated correlation between the two parameters where one observes a repelling \emph{heteroclinic cycle} is  $1$ (precision: $10^{-5}$), as the reader may check  in Section \ref{s:Heteroclinic} and Appendix \ref{apend1}.
We simulate the dynamics in all hyperbolic regions associated to the DZ \emph{singularity} in Figure \ref{subplots}. Section \ref{DandC} finishes this article.

\subsection{Biological consequences}

As a direct consequence of Theorem \ref{th: mainA}, we are able to locate three regions in the parameter space $(\mathcal{R}_0, p )$ where the disease persists 
({\bf C},  {\bf D} and {\bf E}), \emph{i.e.} there is a set with positive Lebesgue measure (in the phase space) whose $\omega$-limit  is the endemic equilibrium.  In regions {\bf C} and {\bf D}, trajectories starting ``below'' $W^s(E_0^p)$ tend to $E_2^p$. In region {\bf E}, there is a repelling periodic solution $\mathcal{C}$ (arising from the \emph{Hopf bifurcation}) which is responsible for the maintenance of an endemic region where the disease persists.

 In region {\bf E}, if we assume seasonality in the transmission rate $\beta$, the associated flow exhibits a strange repeller and hyperbolic horseshoes, as a consequence of the works~\cite{CarvalhoRodrigues2022,WangYoung2003}. The \emph{heteroclinic cycle} $\mathcal{H}$ is a repeller. 
In the regions {\bf F}, {\bf G} and  {\bf H}, the high vaccination does not play a major role in the elimination of the  disease because the carrying capacity $A$ is small compared with the number of \emph{Infected} individuals that are being generated.

\section{The case without vaccination $(p=0)$} \label{p0}

For the sake of completeness, we analyze model \eqref{modelo2} considering $p = 0$:

\begin{equation}
\label{no_vaccine}
\dot{x} = f_{(\mathcal{R}_0 , 0)}(x) \quad \Leftrightarrow \quad
\begin{cases}
&\dot{S} =  S(A-S) - \beta I S \\
\\
&\dot{I} =  \beta IS - \left(\sigma+g\right) I.
\end{cases}
\end{equation}

\bigskip

The disease-free equilibria of system \eqref{no_vaccine} are obtained by imposing $I=0$. Then we get:

$$
E_0 \equiv E_0^0= \left(0,0 \right) \quad \text{and} \quad E_1 \equiv E_1^0= \left(A,0 \right).
$$

\bigskip

\noindent With respect to the endemic equilibrium, we know that:

$$
E_2 = \left(\dfrac{\sigma + g}{\beta},\dfrac{A\beta-\left(\sigma+g\right)}{\beta^2}\right) \overset{\eqref{R0}}{=} \left(\dfrac{A}{\mathcal{R}_0}, \dfrac{A}{\beta} \left(1-\dfrac{1}{\mathcal{R}_0} \right)\right).
$$

\medskip

 The formal equilibrium $E_2\equiv E_2^0$ lies in the interior of the first quadrant when $\mathcal{R}_0 > 1$. 

\bigskip

The jacobian matrix of the vector field associated to (\ref{no_vaccine}) at a general point $E = (S, I) \in (\mathbb{R}_0^+)^2$ is given by:
 
\begin{equation}
\label{jacob}
\begin{array}{lcl}
J(E)=\left(\begin{array}{cc}
-\beta I + A - 2S & - \beta S \\ 
\\
\beta I & \beta S - \left(\sigma+g\right)
\end{array}\right)
\end{array}.
\end{equation}

\noindent Evaluating $J(E)$ at $E_0$, $E_1$ and $E_2$ we have:
 
$$
J(E_0)=\left(\begin{array}{cc}
A  & 0 \\ 
\\
0 & - \left(\sigma+g\right)
\end{array}\right) \qquad J(E_1)=\left(\begin{array}{cc}
-A  & -A\beta \\ 
\\
0 & A\beta - \left(\sigma+g\right)
\end{array}\right) .
$$
and
$$
J(E_2)=\left(\begin{array}{cc}
-\dfrac{\sigma+g}{\beta}  & - \left(\sigma+g\right) \\ 
\\
\dfrac{A\beta - \left(\sigma+g\right)}{\beta} & 0
\end{array}\right) ,
$$

\smallskip

\noindent respectively.  \\

\begin{lemma} \label{E0E1_sta}
System \eqref{no_vaccine} exhibits: \\

\begin{enumerate}
\item two disease-free equilibria, $E_0$ and \,$E_1$, such that:

\begin{enumerate}

\smallskip

\item for  all $\mathcal{R}_0 \in \RR^+$, $E_0$ is a saddle;
 
\smallskip

\item if \,$\mathcal{R}_0 < 1$, then \,$E_1$ is a sink. If \,$\mathcal{R}_0 > 1$, then \,$E_1$ is a saddle;

\smallskip

\item at \,$\mathcal{R}_0 = 1$, \,$E_1$ undergoes a transcritical bifurcation with $E_2$.\\
\end{enumerate} 

\smallskip

\item an endemic equilibrium $E_2$ such that: \\

\begin{enumerate}
 
\item if \,$1<\mathcal{R}_0 < 1+\frac{1}{4\beta}$, then \,$E_2$ is a stable node; 

\smallskip

\item if \,$\mathcal{R}_0 > 1 + \frac{1}{4\beta}$, then \,$E_2$ is a stable focus;

\smallskip

\item at \,$\mathcal{R}_0 = 1+\frac{1}{4\beta}$, \,$E_2$ undergoes a Belyakov transition. \\
\end{enumerate}
\end{enumerate} 
\end{lemma}

\begin{proof}
\begin{enumerate}
\item 
The eigenvalues of $J(E_0)$ are \,$-\left(\sigma+g\right) < 0$ and $ A>0$. Since the eigenvalues of $J(E_0)$ have real part with opposite signs, then $E_0$ is a saddle for  all $\mathcal{R}_0 \in \RR^+$.
 The eigenvalues of $J(E_1)$ are $A \beta - \left(\sigma+g\right)$ and $ -A<0$. If

\begin{eqnarray*}
A\beta - \left(\sigma+g\right) < 0 \quad \Leftrightarrow \quad \dfrac{A\beta}{\sigma + g} < 1 \quad \overset{\eqref{R0}}{\Leftrightarrow} \quad  \mathcal{R}_0 <1 , \label{E1stable}
\end{eqnarray*}
then $E_1$ is a sink. If \,$\mathcal{R}_0 > 1$, then $E_1$ is a saddle. 
Therefore, in the vertical direction, $E_1$ interchanges its stability with $E_2$ from stable to unstable  at $\mathcal{R}_0 = 1$ (\emph{transcritical bifurcation}). \\

\item The eigenvalues of $J(E_2)$ are given by:

$$\lambda_1 \coloneqq \dfrac{- \left(\sigma+g\right) - \sqrt{\Delta}}{2\beta} \qquad \text{and} \qquad \lambda_2 \coloneqq \dfrac{- \left(\sigma+g\right) + \sqrt{\Delta}}{2\beta},$$
where $\Delta = -4\left(\sigma+g\right)\left[ \left(-\beta - \dfrac{1}{4} \right)g + A\beta^2 - \sigma \beta - \dfrac{\sigma}{4} \right]$. Since $\beta >0$ and $\sigma+g>0$, it is easy to verify that $\lambda_1$ has negative real part. The eigenvalue $\lambda_2$ has negative real part if and only if

\begin{eqnarray}
\nonumber && -\left( \sigma + g \right) + \sqrt{\Delta} < 0 \\
\nonumber && \\
\nonumber \Leftrightarrow&& \sqrt{\Delta} < \sigma + g  \\
\nonumber && \\
\nonumber \Leftrightarrow &&
\begin{cases} &\Delta \geq 0 \\
\\
&\Delta < \left( \sigma + g \right)^2
\end{cases} \\
\nonumber &&\\
\nonumber \Leftrightarrow && 
\begin{cases} & -4\left(\sigma+g\right)\left[ \left(-\beta - \dfrac{1}{4} \right)g + A\beta^2 - \sigma \beta - \dfrac{\sigma}{4} \right] \geq 0 \\
\\
&-4\left(\sigma+g\right)\left[ \left(-\beta - \dfrac{1}{4} \right)g + A\beta^2 - \sigma \beta - \dfrac{\sigma}{4} \right] < \left( \sigma + g \right)^2
\end{cases} \\ 
\nonumber &&\\
\nonumber \Leftrightarrow &&
\begin{cases} & \left(-\beta - \dfrac{1}{4} \right)g + A\beta^2 - \sigma \beta - \dfrac{\sigma}{4} \leq 0 \\
\\
& -4\left[ \left(-\beta - \dfrac{1}{4} \right)g + A\beta^2 - \sigma \beta - \dfrac{\sigma}{4} \right] < \sigma + g
\end{cases} \\ 
\nonumber &&\\
\nonumber \Leftrightarrow &&
\begin{cases} & 4A\beta^2 - \left(4 \sigma \beta + \sigma \right) - \left( 4\beta + 1 \right)g \leq 0 \\
\\
& g - A\beta +\sigma  < 0 
\end{cases} \\ 
\nonumber &&\\
\nonumber \Leftrightarrow &&
\begin{cases} & 4\beta \mathcal{R}_0 - \dfrac{\sigma\left(4 \beta + 1 \right)}{\sigma + g} - \dfrac{g\left(4 \beta + 1 \right)}{\sigma + g} \leq 0 \\
\\
& \mathcal{R}_0 > 1
\end{cases} \\ 
\nonumber &&\\
\nonumber \Leftrightarrow &&
\begin{cases} & \mathcal{R}_0 \leq  1 + \dfrac{1}{4\beta} \\
\\
& \mathcal{R}_0 > 1 .
\end{cases} 
\end{eqnarray}

\medskip

If $1 < \mathcal{R}_0 < 1+\frac{1}{4\beta}$, then $\lambda_2 < 0$ and $E_2$ is a stable node. In an analogous way, if $\mathcal{R}_0 > 1 + \frac{1}{4\beta}$, then $\Delta < 0$ and $E_2$ is a stable focus. Hence, $E_2$ evolves from stable node to stable focus at $\mathcal{R}_0 = 1+\frac{1}{4\beta}$ (\emph{Belyakov transition}).
\end{enumerate}
\end{proof}

In Figure \ref{without_p}, one observes the scheme of the equilibria stability of system \eqref{no_vaccine} for different values of $\mathcal{R}_0$.  The {\it transcritical bifurcation} at  $\mathcal{R}_0 = 1$ represents the threshold for the existence of the endemic disease in the population, agreeing well with the empirical belief.
\begin{figure}[h!]
\center
\includegraphics[scale=0.30]{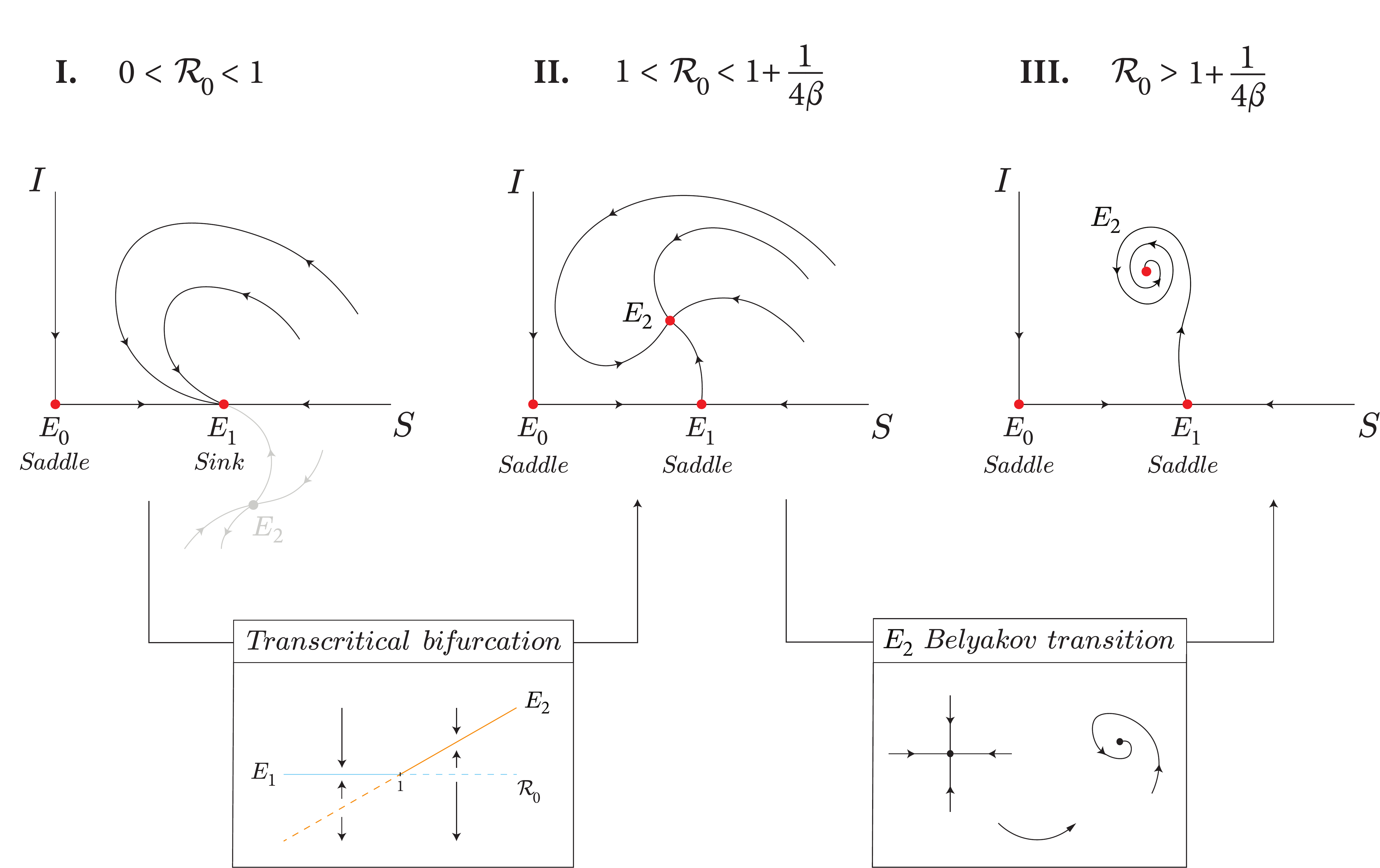}
\caption{Phase diagram of \eqref{no_vaccine} for different values of $\mathcal{R}_0$. {\bf I.} $E_1$ is a \emph{global attractor} when restricted to the closure of the first quadrant. {\bf II.} Besides the $E_0$ and $E_1$, there exists a stable node $E_2$. {\bf III.} Besides the $E_0$ and $E_1$, there exists a stable focus $E_2$. From \textbf{I.} to \textbf{II.} $E_1$ undergoes a \emph{transcritical bifurcation} at $\mathcal{R}_0 = 1$. From \textbf{II.} to \textbf{III.} $E_2$ evolves a \emph{Belyakov transition} at $\mathcal{R}_0 = 1+\frac{1}{4\beta}$.}
\label{without_p}
\end{figure}


\section{The case with constant vaccination $(p>0)$} \label{pDIF0}

\noindent We  analyze  \eqref{modelo2} by considering $p > 0$ and $\mathcal{R}_0 > 1$ (see {\bf (C3)}):

\begin{equation}
\label{with_vaccine}
\dot{x} =  f_{(\mathcal{R}_0 , 0)}(x) \quad \Leftrightarrow \quad
\begin{cases} &\dot{S} =  S(A-S) - \beta I S - pm \\
\\
&\dot{I} =  \beta IS - \left(\sigma+g\right) I.
\end{cases}
\end{equation}

\bigskip

For $A$ and $m$ fixed, we settle the following maps that will be used throughout this text:
\begin{equation}
\label{pes_eq}
p_{SN}(\mathcal{R}_0) \coloneqq \dfrac{A^2}{4m} \quad \text{,} \quad p_T(\mathcal{R}_0) \coloneqq \dfrac{A^2}{m} \dfrac{\left(\mathcal{R}_0-1\right)}{\mathcal{R}_0^2}  \quad \text{and} \quad p_H(\mathcal{R}_0) \coloneqq \dfrac{A^2}{m\mathcal{R}_0^2},
\end{equation}

\medskip

\noindent for $\mathcal{R}_0 > 1$. It is easy to verify that:\\

\begin{itemize}
\item  $p_{SN}$ is a constant map and \\
\item if $1 < \mathcal{R}_0 < 2$, then $p_H (\mathcal{R}_0)>p_{SN}(\mathcal{R}_0)$, which does not make sense (because $E_2^p$ does not exist in the first quadrant as a consequence of Lemmas \ref{pes} and \ref{E2_EXIS}). This is the reason why we consider the set $[2, +\infty[$ as the domain of $p_H$.\\
\end{itemize}
 
\subsection{Preparatory section}
 
We state a preliminary result that will be used in the sequel.

\begin{lemma} \label{pes}
The following statements hold for the maps given in  \eqref{pes_eq}:\\

\begin{enumerate}
\item If \,$\mathcal{R}_0 > 1$ and \,$\mathcal{R}_0 \neq 2$, then \,$p_T  (\mathcal{R}_0)< p_{SN} (\mathcal{R}_0)$;
\medskip
\item   \,$p_H (2)= p_T(2) = p_{SN}(2)$;
\medskip
\item If \,$\mathcal{R}_0 > 2$, then  \,$p_H (\mathcal{R}_0) < p_T(\mathcal{R}_0)<p_{SN}(\mathcal{R}_0)$.\\
\end{enumerate}
\end{lemma}

\begin{proof}
\begin{enumerate}
\item 
From \eqref{pes_eq}, it follows that:
\begin{eqnarray}
\nonumber && p_T(\mathcal{R}_0) < p_{SN} (\mathcal{R}_0) \\
\nonumber && \\
\nonumber \Leftrightarrow&& \dfrac{A^2}{m} \dfrac{\left(\mathcal{R}_0-1\right)}{\mathcal{R}_0^2} < \dfrac{A^2}{4m}\\
\nonumber && \\
\nonumber \Leftrightarrow && \dfrac{\mathcal{R}_0-1}{\mathcal{R}_0^2} < \dfrac{1}{4} \\
\nonumber \\
\nonumber \Leftrightarrow&& \mathcal{R}_0^2 - 4\mathcal{R}_0 + 4 > 0\\
\nonumber \\
\nonumber \Leftrightarrow&& (\mathcal{R}_0-2)^2 > 0. \\
\end{eqnarray}
 \item This item follows if we evaluate $p_H$, $p_T$ and $p_{SN}$  at $\mathcal{R}_0 = 2$.\\
\item 

From \eqref{pes_eq}, we have:
\begin{eqnarray}
\nonumber && p_H  (\mathcal{R}_0)< p_T  (\mathcal{R}_0) \\
\nonumber && \\
\nonumber \Leftrightarrow&& \dfrac{A^2}{m\mathcal{R}_0^2} < \dfrac{A^2}{m} \dfrac{\left(\mathcal{R}_0-1\right)}{\mathcal{R}_0^2}  \\
\nonumber && \\
\nonumber \Leftrightarrow && \mathcal{R}_0-1 > 1 \\
\nonumber \\
\nonumber \Leftrightarrow&&\mathcal{R}_0  > 2. \\ 
\end{eqnarray}

The remaining inequality follows from the previous item and transitivity.
\end{enumerate}
\end{proof}

\subsection{Saddle-node bifurcation}
\label{s:sn} The general jacobian matrix of   (\ref{with_vaccine}) at $E^p = (S^p, I^p) \in (\mathbb{R}_0^+)^2$ coincides with that of \eqref{jacob}. At $E_0^p$, $E_1^p$ and $E_2^p$, the matrix \eqref{jacob} takes the form:
$$
J(E_0^p)=\left(\begin{array}{cc}
A - 2S_0^p  & -\beta S_0^p \\ 
\\
0 & \beta S_0^p - \left(\sigma+g\right)
\end{array}\right), \qquad J(E_1^p)=\left(\begin{array}{cc}
A - 2S_1^p & -\beta S_1^p \\ 
\\
0 & \beta S_1^p - \left(\sigma+g\right)
\end{array}\right) 
$$

and

\begin{eqnarray}
\label{JE2_16}
J(E_2^p)=\left(\begin{array}{cc}
\dfrac{pm\beta^2-\left(\sigma+g\right)^2}{\beta\left(\sigma+g\right)}  & - \left(\sigma+g\right) \\ 
\\
\dfrac{-pm\beta^2 + A\left(\sigma+g\right) \beta - \left(\sigma+g\right)^2}{\beta\left(\sigma+g\right)} & 0
\end{array}\right),
\end{eqnarray}
 respectively. \\ 


\begin{lemma} 
The following statements hold for model \eqref{with_vaccine}:\\

\begin{enumerate}
\item If \,$ p > p_{SN}(\mathcal{R}_0 )$, then $E_0^p$ and $E_1^p$ do not exist in the first quadrant of $(S,I)$. \\
\item If \,$p_T (\mathcal{R}_0 )< p < p_{SN}(\mathcal{R}_0 )$, then:

\smallskip

\begin{enumerate}

\item $E_0^p$ is a saddle and $E_1^p$ is a sink for \,$1 < \mathcal{R}_0 < 2$;
\smallskip
\item $E_0^p$ is a source and $E_1^p$ is a saddle for \,$\mathcal{R}_0 > 2$.\\
\end{enumerate}

\smallskip

\item If \,$p<p_T(\mathcal{R}_0 )$ and \,$\mathcal{R}_0 > 1$, then \,$E_0^p$ and \,$E_1^p$ are saddles. \\
\end{enumerate}
\end{lemma}

\begin{proof}

The eigenvalues of $J(E_0^p)$ are $A-2S_0^p$ and $\beta S_0^p - (\sigma+g)$. We know that

\begin{eqnarray} 
\nonumber && A-2S_0^p > 0 \\
\nonumber && \\
\nonumber &\Leftrightarrow& \displaystyle \dfrac{2A - 2A + 2\sqrt{A^2 - 4pm}}{2} > 0 \\
\nonumber && \\
\nonumber &\Leftrightarrow& \displaystyle \sqrt{A^2 - 4pm} > 0 \quad \text{if $p \leq p_{SN}(\mathcal{R}_0 )$}
\end{eqnarray}

\begin{enumerate}
\item 
If $p> p_{SN}(\mathcal{R}_0 )$, then $E_0^p$ and $E_1^p$ do not exist. \\
\item We may write: 
\begin{eqnarray} 
\nonumber && \beta S_0^p - (\sigma+g)  < 0 \\
\nonumber && \\
\nonumber &\Leftrightarrow& S_0^p < \dfrac{\sigma + g}{\beta} \\
\nonumber && \\
\nonumber & \overset{\eqref{S0_ref}}{\Leftrightarrow}& \displaystyle \dfrac{A-\sqrt{A^2 - 4pm}}{2} <  \dfrac{\sigma + g}{\beta} \\
\nonumber && \\
\nonumber & \Leftrightarrow& \displaystyle \sqrt{A^2 - 4pm} > A - \dfrac{2\left(\sigma+g\right)}{\beta} \\
\nonumber && \\
\nonumber & \Leftrightarrow& \displaystyle \sqrt{A^2 - 4pm} > \dfrac{A\mathcal{R}_0}{\mathcal{R}_0}  - \dfrac{2A}{\mathcal{R}_0} \\
\nonumber && \\
\nonumber &\Leftrightarrow & \displaystyle \sqrt{A^2 - 4pm} > \dfrac{A\left(\mathcal{R}_0 - 2\right)}{\mathcal{R}_0}  \\
\nonumber && \\
\nonumber &\Leftrightarrow &
\begin{cases} &A^2 - 4pm \geq 0 \quad , \qquad 1 < \mathcal{R}_0 < 2 \\
\\
&A^2 - 4pm > \dfrac{A^2\left(\mathcal{R}_0 - 2\right)^2}{\mathcal{R}_0^2} \quad , \qquad \mathcal{R}_0 \geq 2
\end{cases} \\
\nonumber &&\\
\nonumber &\overset{\eqref{pes_eq}}{\Leftrightarrow}& 
\begin{cases} &p \leq p_{SN} \quad , \qquad 1 < \mathcal{R}_0 < 2 \\
\\
& p < \dfrac{A^2\left(\mathcal{R}_0 - 1\right)}{m\mathcal{R}_0^2} \quad , \qquad \mathcal{R}_0 \geq 2
\end{cases} \\
\nonumber &&\\
&\overset{\eqref{pes_eq}}{\Rightarrow}& 
\begin{cases} &p < p_{SN}(\mathcal{R}_0) \quad , \qquad 1 < \mathcal{R}_0 < 2 \\
\\
&p < p_T(\mathcal{R}_0) \quad , \qquad \mathcal{R}_0 > 2 \label{E0_saddle}
\end{cases},
\end{eqnarray}
then the eigenvalues of $J(E_0^p)$ have real part with opposite signs and $E_0^p$ is a saddle. On the other hand, if

\begin{eqnarray}
\nonumber && \beta S_0 - (\sigma+g)  > 0 \\
\nonumber && \\
\nonumber &\Leftrightarrow& S_0 > \dfrac{\sigma + g}{\beta} \\
\nonumber && \\
\nonumber & \Leftrightarrow&  \sqrt{A^2 - 4pm} < \dfrac{A\left(\mathcal{R}_0 -2 \right)}{\mathcal{R}_0} \\
\nonumber && \\
\nonumber &\Leftrightarrow &0 \leq A^2 - 4pm < \dfrac{A^2 \left(\mathcal{R}_0 - 2 \right)^2}{\mathcal{R}_0^2} \quad , \qquad \mathcal{R}_0 > 2 \\
\nonumber && \\
&\overset{\eqref{pes_eq}}{\Rightarrow}& p_T(\mathcal{R}_0) < p < p_{SN}(\mathcal{R}_0) \quad , \qquad \mathcal{R}_0 > 2 \label{E0_source_E1_saddle},
\end{eqnarray}
then the eigenvalues of $J(E_0^p)$ have real part with positive signs, and $E_0^p$ is a source. Moreover, the eigenvalues of $J(E_1^p)$ are $A - 2S_1^p < 0$ and $\beta S_1^p - \left(\sigma+g\right)$. Hence, if

\begin{eqnarray}
\nonumber && \beta S_1^p - (\sigma+g)  < 0 \\
\nonumber && \\
\nonumber &\Leftrightarrow& S_1^p < \dfrac{\sigma + g}{\beta} \\
\nonumber && \\
\nonumber &\Leftrightarrow& \sqrt{A^2 - 4pm} < \dfrac{A\left(2 - \mathcal{R}_0 \right)}{\mathcal{R}_0} \\
\nonumber && \\
\nonumber &\Leftrightarrow& 0 \leq A^2 - 4pm < \dfrac{A^2 \left(2 - \mathcal{R}_0 \right)^2}{\mathcal{R}_0^2} \quad , \qquad 1 < \mathcal{R}_0 < 2 \\
\nonumber && \\
&\overset{\eqref{pes_eq}}{\Rightarrow}& p_T(\mathcal{R}_0) < p < p_{SN}(\mathcal{R}_0) \quad , \qquad 1 < \mathcal{R}_0 < 2 , \label{E1_sink}
\end{eqnarray}
then both eigenvalues of $J(E_1^p)$ have negative real part and $E_1^p$ is a sink. If
\begin{eqnarray}
\nonumber && \beta S_1^p - (\sigma+g)  > 0 \\
\nonumber && \\
\nonumber &\Leftrightarrow& S_1^p > \dfrac{\sigma + g}{\beta} \\
\nonumber && \\
\nonumber &\Leftrightarrow& \sqrt{A^2 - 4pm} > \dfrac{A\left(2 - \mathcal{R}_0 \right)}{\mathcal{R}_0} \\
\nonumber && \\
\nonumber &\Leftrightarrow &
\begin{cases} &A^2 - 4pm > \dfrac{A^2\left(2 - \mathcal{R}_0 \right)^2}{\mathcal{R}_0^2} \quad , \qquad 1 < \mathcal{R}_0 \leq 2 \\
\\
&A^2 - 4pm \geq 0 \quad , \qquad \mathcal{R}_0 > 2
\end{cases} \\
\nonumber &&\\
\nonumber &\Leftrightarrow &
\begin{cases} &p < \dfrac{A^2}{m}\left(\dfrac{\mathcal{R}_0 - 1}{\mathcal{R}_0^2}\right) \quad , \qquad 1 < \mathcal{R}_0 \leq 2 \\
\\
&p \leq \dfrac{A^2}{4m} \quad , \qquad \mathcal{R}_0 > 2
\end{cases} \\
\nonumber &&\\
&\overset{\eqref{pes_eq}}{\Rightarrow}& 
\begin{cases} &p < p_T(\mathcal{R}_0) \quad , \qquad 1 < \mathcal{R}_0 < 2 \\
\\
&p < p_{SN}(\mathcal{R}_0) \quad , \qquad \mathcal{R}_0 > 2 , \label{E1_saddle} 
\end{cases}
\end{eqnarray}
then the eigenvalues of $J(E_1^p)$ have real part with opposite signs and $E_1^p$ is a saddle. 
Therefore, from \eqref{E0_saddle} and \eqref{E1_sink}, we conclude that:\\
\begin{enumerate}
\item  if \,$p_T (\mathcal{R}_0)< p < p_{SN}(\mathcal{R}_0)$ and \,$1 < \mathcal{R}_0 < 2$, then $E_0^p$ is a saddle and $E_1^p$ is a sink; \\
\item if \,$p_T(\mathcal{R}_0) < p < p_{SN}(\mathcal{R}_0)$, then $E_0^p$ is a source and $E_1^p$ is a saddle, for \,$\mathcal{R}_0 > 2$, under conditions \eqref{E0_source_E1_saddle} and \eqref{E1_saddle}. \\
\end{enumerate}
\item
From \eqref{E0_saddle} and \eqref{E1_saddle}, we conclude that if \,$p < p_T(\mathcal{R}_0)$, then $E_0^p$ and $E_1^p$ are saddles.
\end{enumerate}
\end{proof} 

 \subsection{Transcritical bifurcation}
\label{s:transcritical}
\begin{lemma}\label{E2_EXIS}
The endemic equilibrium $E_2^p$ undergoes a transcritical bifurcation at $p = p_T(\mathcal{R}_0)$ and lies in the interior of the first quadrant if \,$p <p_T(\mathcal{R}_0)$.  
\end{lemma}

\begin{proof}
The equilibrium $E_2^p$ lies in the interior of the first quadrant if both $S_2^p$ and $I_2^p$ are positive. It is clear that $S_2^p =\dfrac{\sigma + g}{\beta}> 0$. Since $\beta^2(\sigma + g) > 0$, we have

$$I_2^p > 0 \quad \Leftrightarrow \quad -pm\beta^2 + A\left(\sigma+g\right)\beta-\left(\sigma+g\right)^2 > 0.$$

Solving the  previous inequality in order to $p$, one gets:

\begin{eqnarray}
\nonumber &&-pm\beta^2 + A\left(\sigma+g\right)\beta-\left(\sigma+g\right)^2 > 0 \\
\nonumber && \\
\nonumber \Leftrightarrow&& p < \dfrac{A\left(\sigma+g\right)\beta-\left(\sigma+g\right)^2}{m\beta^2} \\
\nonumber && \\
\nonumber \Leftrightarrow&& p < \dfrac{A\left(\sigma+g\right)^2\beta}{m\left(\sigma+g\right)\beta^2}-\dfrac{\left(\sigma+g\right)^2A^2}{A^2m\beta^2} \\
\nonumber && \\
\nonumber \Leftrightarrow&& p < \dfrac{\left(\sigma+g\right)^2 A^2}{m\beta^2A^2}\mathcal{R}_0 - \dfrac{1}{\mathcal{R}_0^2}\dfrac{A^2}{m} \\
\nonumber && \\
\nonumber \Leftrightarrow&& p < \dfrac{1}{\mathcal{R}_0} \dfrac{A^2}{m} - \dfrac{1}{\mathcal{R}_0^2}\dfrac{A^2}{m} \\
\nonumber && \\
\nonumber \Leftrightarrow&& p < \dfrac{A^2}{m} \dfrac{\left(\mathcal{R}_0 - 1\right)}{\mathcal{R}_0^2} \\
\nonumber && \\
\nonumber \overset{\eqref{pes_eq}}{\Leftrightarrow}&& p < p_T(\mathcal{R}_0)
\end{eqnarray}
since $\mathcal{R}_0 > 1$. Hence, $E_2^p$ undergoes a \emph{trancritical bifurcation} along the line  $p = p_T(\mathcal{R}_0)$ and  interchanges its stability with $E_2^p$ (if $\mathcal{R}_0<2$) and $E_1^p$ (if $\mathcal{R}_0>2$).
\end{proof}

\subsection{Belyakov transition}
\label{s:Belyakov}
We settle the following functions that will be used throughout this text:

\begin{eqnarray}
\nonumber \Delta_2(p) &\coloneqq& m^2 p^2 \beta^4 + 4pm\left(\sigma+g\right)^2\beta^3 - \left[4A\left(\sigma+g\right)+2pm\right]\left(\sigma+g\right)^2\beta^2 \\
\nonumber &&+ 4\left(\sigma+g\right)^4\beta+\left(\sigma+g\right)^4 \\
\nonumber && \\
\nonumber &=& m^2 \beta^4 p^2 + 2m\beta^2\left( 2\beta-1 \right) \left( \sigma + g \right)^2 p + \left(4\beta+1\right)\left(\sigma+g\right)^4 \\
&&- 4A\left(\sigma+g\right)^3\beta^2 , \label{delta2} \\
\nonumber && \\
p_{B_t}^{(1)}\left(\mathcal{R}_0 \right) &\coloneqq& \dfrac{\left(-2\beta+1-2\displaystyle\sqrt{\beta\left(\mathcal{R}_0 + \beta - 2\right)}\right)A^2}{m\mathcal{R}_0^2} \label{pB1} , \\
\nonumber && \\
p_{B_t}^{(2)}\left(\mathcal{R}_0 \right) &\coloneqq& \dfrac{\left(-2\beta+1+2\displaystyle\sqrt{\beta\left(\mathcal{R}_0 + \beta - 2\right)}\right)A^2}{m\mathcal{R}_0^2} \label{pB2} ,
\end{eqnarray}

\medskip

\noindent for \,$\beta\left(\mathcal{R}_0 + \beta - 2 \right) > 0 \Leftrightarrow \beta > 2 - \mathcal{R}_0$, where \,$p_{B_t}^{(1)}(\mathcal{R}_0)$ and \,$p_{B_t}^{(2)}(\mathcal{R}_0)$ are the square roots of \,$\Delta_2(p)$ written as function of $\mathcal{R}_0$. It is easy to verify that \,$p_{B_t}^{(1)} (\mathcal{R}_0)< p_{B_t}^{(2)}(\mathcal{R}_0)$. 

\begin{lemma} \label{PB2_proof}
If \,$\beta \geq \frac{1}{2}$ and \,$\mathcal{R}_0 > 1 + \frac{1}{4\beta}$, then \,$p_{B_t}^{(2)} (\mathcal{R}_0)> 0$.
\end{lemma}

\begin{proof}
From \eqref{pB2} we know that if

\begin{eqnarray*}
\nonumber && \dfrac{\left(-2\beta+1+2\displaystyle\sqrt{\beta\left(\mathcal{R}_0 + \beta - 2\right)}\right)A^2}{m\mathcal{R}_0^2}  > 0 \\
\nonumber && \\
\nonumber &\Leftrightarrow& \left(-2\beta+1+2\displaystyle\sqrt{\beta\left(\mathcal{R}_0 + \beta - 2\right)}\right)A^2 > 0 \\
\nonumber && \\
\nonumber &\Leftrightarrow& 2\displaystyle\sqrt{\beta\left(\mathcal{R}_0 + \beta - 2\right)} > 2\beta - 1\\
\nonumber && \\
\nonumber &\Leftrightarrow& \displaystyle\sqrt{\beta\left(\mathcal{R}_0 + \beta - 2\right)} > \beta - \dfrac{1}{2} \\
\nonumber && \\
\nonumber &\Leftrightarrow &
\begin{cases} & \beta - \dfrac{1}{2} \geq 0 \quad , \qquad \beta\left(\mathcal{R}_0 + \beta - 2 \right) > \left[\beta-\dfrac{1}{2}\right]^2 \\
\\
& \beta - \dfrac{1}{2} < 0 \quad , \qquad \beta\left(\mathcal{R}_0 + \beta - 2 \right) \geq 0
\end{cases} \\
\nonumber &&\\
\nonumber &\Leftrightarrow &
\begin{cases} & \beta \geq \dfrac{1}{2} \quad , \qquad  \beta \mathcal{R}_0 - 2\beta > -\beta + \dfrac{1}{4} \\
\\
&\beta < \dfrac{1}{2} \quad , \qquad \beta^2 + \mathcal{R}_0 \beta - 2\beta \geq 0
\end{cases} \\
\nonumber &&\\
\nonumber &\Leftrightarrow &
\begin{cases} & \beta \geq \dfrac{1}{2} \quad , \qquad  \beta \left(\mathcal{R}_0 - 1 \right) > \dfrac{1}{4} \\
\\
&\beta < \dfrac{1}{2} \quad , \qquad \beta \geq 2 - \mathcal{R}_0
\end{cases} \\
\nonumber &&\\
&\Leftrightarrow &
\begin{cases} & \beta \geq \dfrac{1}{2} \quad , \qquad  \mathcal{R}_0 > 1+\dfrac{1}{4\beta} \\
\\
&\beta < \dfrac{1}{2} \quad , \qquad \mathcal{R}_0 \geq 2 - \beta
\end{cases} ,
\end{eqnarray*}
then $p_{B_t}^{(2)} (\mathcal{R}_0)> 0$. If \,$\mathcal{R}_0 = 1+\frac{1}{4\beta}$, then \,$p_{B_t}^{(2)}(\mathcal{R}_0) = 0$ and if \,$\mathcal{R}_0 = 2$, then \,$p_{B_t}^{(2)}(\mathcal{R}_0) = p_{SN}(\mathcal{R}_0)$. Hence we conclude that if \,$\beta \geq \frac{1}{2}$ and $\mathcal{R}_0 > 1 + \frac{1}{4\beta}$, then \,$p_{B_t}^{(2)}(\mathcal{R}_0) > 0$. 
\end{proof}

Before we analyze the stability of $E_2^p$, we remind the readers that $E_2^p$ exists in the first quadrant when $p < p_T(\mathcal{R}_0)$ (by Lemma \ref{E2_EXIS}).

\begin{lemma}\label{lemma_grande}
The following statements hold for model \eqref{with_vaccine}:

\bigskip

\begin{enumerate}
\item If \,$p_{B_t}^{(2)}(\mathcal{R}_0)<p<p_T(\mathcal{R}_0)$, then \,$\Delta_2(p) > 0$ and $E_2^p$ is:  \label{ponto1}

\bigskip

\begin{enumerate}
\item stable node for \,$ \mathcal{R}_0<2$;
\item unstable node for \,$p > p_H(\mathcal{R}_0)$.
\end{enumerate}

\bigskip
\smallskip

\item If \,$p<p_{B_t}^{(2)}(\mathcal{R}_0)$, then \,$\Delta_2(p) < 0$ and \,$E_2^p$ is: \label{ponto2}

\bigskip

\begin{enumerate}
\item stable focus for \,$p < p_H(\mathcal{R}_0)$;
\item unstable focus for \,$p > p_H(\mathcal{R}_0)$.
\end{enumerate}
\end{enumerate}

\medskip

Hence, $E_2^p$ undergoes a Belyakov transition at \,$p=p_{B_t}^{(2)}(\mathcal{R}_0)$.
\end{lemma}

\medskip

\begin{proof}

\begin{enumerate}

With respect to \eqref{with_vaccine}, the eigenvalues of the jacobian matrix of $J(E_2^p)$ are:
$$\lambda_1 = \dfrac{p m \beta^2 - \left(\sigma + g \right)^2 - \displaystyle\sqrt{\Delta_2(p)}}{2\beta \left(\sigma + g\right)} \qquad \text{and} \qquad \lambda_2 = \dfrac{p m \beta^2 - \left(\sigma + g \right)^2 + \displaystyle\sqrt{\Delta_2(p)}}{2\beta \left(\sigma + g\right)}.$$
where $\lambda_1 < \lambda_2$. Hence, if $\lambda_1>0$, then $\lambda_2>0$. If $\lambda_2<0$, then $\lambda_1<0$.
We know from \eqref{delta2}, \eqref{pB1}, \eqref{pB2} and Lemma \ref{PB2_proof} that:

\begin{eqnarray*} \label{DELTAS00}
\begin{cases}\,\, \Delta_2 > 0 \,\,, \qquad \text{if} \quad \,0<p<p_{B_t}^{(1)}$ or \,$p_{B_t}^{(2)}<p<+\infty  \\
\\
\,\, \Delta_2 < 0 \,\,, \qquad \text{if} \quad \,p_{B_t}^{(1)}<p<p_{B_t}^{(2)}
\end{cases},
\end{eqnarray*}
 where $p_{B_t}^{(2)} (\mathcal{R}_0)< p_T(\mathcal{R}_0)$.

\bigskip

(\ref{ponto1}) Therefore, if

\begin{eqnarray*}
\nonumber && \lambda_1 > 0 \\
\nonumber && \\
\nonumber &\Leftrightarrow & \dfrac{p m \beta^2 - \left(\sigma + g \right)^2 - \displaystyle\sqrt{\Delta_2(p)}}{2\beta \left(\sigma + g\right)} > 0 \\
\nonumber && \\
\nonumber &\Leftrightarrow & \displaystyle\sqrt{\Delta_2(p)} < p m \beta^2 - \left(\sigma + g \right)^2 \\
\nonumber && \\
\nonumber &\Leftrightarrow &
\begin{cases} &\Delta_2(p) \geq 0 \\
\\
&\Delta_2(p) < \left[p m \beta^2 - \left(\sigma + g \right)^2 \right]^2
\end{cases} \qquad, \quad p m \beta^2 - \left(\sigma + g \right)^2 > 0  \\
\nonumber &&\\
\nonumber &\Leftrightarrow& 
\begin{cases} &p \geq p_{B_t}^{(2)} (\mathcal{R}_0)\\
\\
& p < p_T(\mathcal{R}_0)
\end{cases} \qquad, \quad p > p_H (\mathcal{R}_0) \\
\nonumber &&\\
&\Rightarrow& p_{B_t}^{(2)}(\mathcal{R}_0) < p < p_T(\mathcal{R}_0) \quad, \qquad \text{for} \quad p > p_H (\mathcal{R}_0)
\label{lambdaS_000}
\end{eqnarray*}
then $\lambda_1 > 0 \Rightarrow \lambda_2 > 0$ and $E_2^p$ is an unstable node. Analogously, we proceed in the same way for $\lambda_2 < 0$. If

\begin{eqnarray*}
\nonumber && \lambda_2 < 0 \\
\nonumber && \\
\nonumber &\Leftrightarrow & \dfrac{p m \beta^2 - \left(\sigma + g \right)^2 + \displaystyle\sqrt{\Delta_2(p)}}{2\beta \left(\sigma + g\right)} < 0 \\
\nonumber && \\
\nonumber &\Leftrightarrow & \displaystyle\sqrt{\Delta_2(p)} <  \left(\sigma + g \right)^2 - p m \beta^2 \\
\nonumber && \\
\nonumber &\Leftrightarrow &
\begin{cases} &\Delta_2(p) \geq 0 \\
\\
&\Delta_2(p) < \left[\left(\sigma + g \right)^2 - p m \beta^2 \right]^2
\end{cases} \qquad, \quad \left(\sigma + g \right)^2 - p m \beta^2 > 0  \\
\nonumber &&\\
\nonumber &\Leftrightarrow& 
\begin{cases} &p \geq p_{B_t}^{(2)} (\mathcal{R}_0)\\
\\
& p < p_T(\mathcal{R}_0)
\end{cases} \qquad, \quad p < p_H(\mathcal{R}_0)  \\
\nonumber &&\\
&\Rightarrow& p_{B_t}^{(2)} (\mathcal{R}_0)< p<p_T (\mathcal{R}_0) \quad, \qquad \text{for} \quad p < p_H (\mathcal{R}_0)
\label{lambdaS_000_0}
\end{eqnarray*}

\medskip

then $\lambda_2 < 0 \Rightarrow \lambda_1 < 0$ and $E_2^p$ is a stable node.

\bigskip

(\ref{ponto2}) Now, let us assume that \,$\Delta_2(p) < 0$, which means that \,$p<p_{B_t}^{(2)}(\mathcal{R}_0)$, the eigenvalues are complex and \,$E_2^p$ is a focus. If the real part of the eigenvalues of \,$J(E_2^p)$ is negative (and \,$\beta$ is positive), \emph{i.e.} if

\begin{eqnarray}
\nonumber && pm\beta^2 - \left(\sigma+g\right)^2 < 0 \\
\nonumber && \\
\nonumber \Leftrightarrow && pm\beta^2 < \left(\sigma+g\right)^2 \\
\nonumber && \\
\nonumber \Leftrightarrow&& p < \dfrac{\left(\sigma+g\right)^2}{m\beta^2}  \\
\nonumber && \\
\nonumber \Leftrightarrow&& 
p < \dfrac{A^2}{m\mathcal{R}_0^2} \\
\nonumber && \\
\nonumber \overset{\eqref{pes_eq}}{\Leftrightarrow}&& p < p_H(\mathcal{R}_0),
\end{eqnarray}
then \,$E_2^p$ is a stable focus. Otherwise \,$E_2^p$ is an unstable focus. Also, we conclude that \,$E_2^p$ undergoes a \emph{Belyakov transition} at \,$p=p_{B_t}^{(2)}(\mathcal{R}_0)$, since $E_2^p$ changes its stability from a node to a focus.
\end{enumerate}
\end{proof}

\subsection{Hopf bifurcation}
\label{s:Hopf}

In this section we will find a line in the parameter space $(p, \mathcal{R}_0) $ where   a \emph{Hopf bifurcation} occurs.
\begin{lemma} \label{HOPF_lemma}
If \,$p = p_H(\mathcal{R}_0)$, 
then $E_2^p$ undergoes a subcritical Hopf bifurcation associated to an unstable periodic solution $\mathcal{C}$ (as $\mathcal{R}_0$ increases).
\end{lemma}

\begin{figure}[h!]
\vspace{2.5cm}
\center
\includegraphics[scale=0.43]{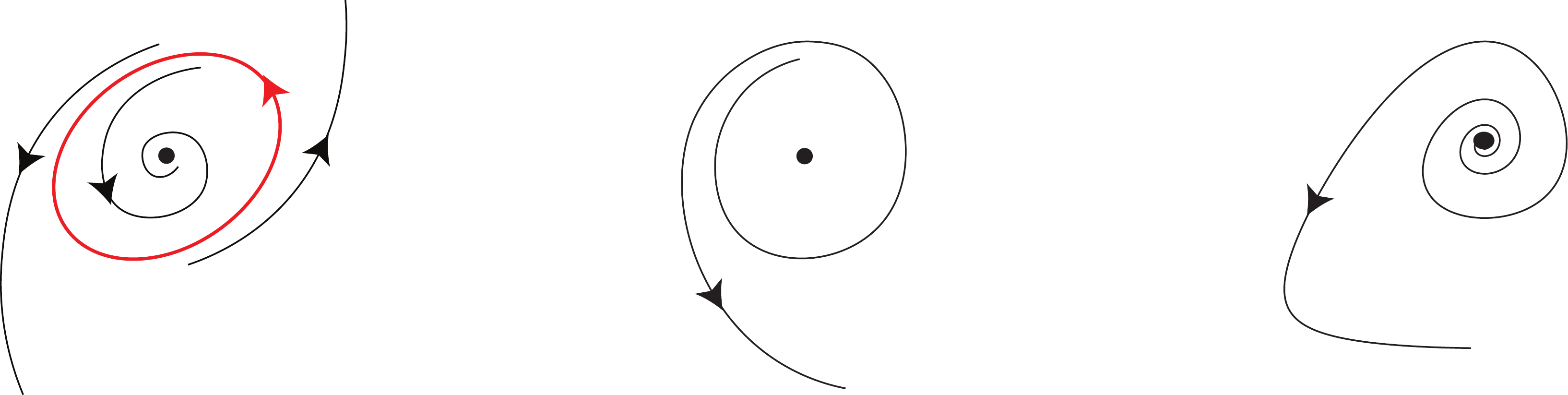}
\caption{Phase portraits of system \eqref{with_vaccine} when $E_2^p$ undergoes a \emph{subcritical Hopf bifurcation as $\mathcal{R}_0$ increases.}}
\label{HOPF_bifurcation}
\end{figure}

\begin{proof}
Let the eigenvalues of $J(E_2^p)$ given by 

$$\lambda_1 = \dfrac{p m \beta^2 - \left(\sigma + g \right)^2 - \displaystyle\sqrt{\Delta_2(p)}}{2\beta \left(\sigma + g\right)} \qquad \text{and} \qquad \lambda_2 = \dfrac{p m \beta^2 - \left(\sigma + g \right)^2 + \displaystyle\sqrt{\Delta_2(p)}}{2\beta \left(\sigma + g\right)}.$$

From~\cite[pp.~151--153, Theorem 3.4.2 (adapted)]{GuckenheimerHolmes1983} we know that the \emph{Hopf bifurcation} occurs when the following conditions hold:

\medskip

\begin{enumerate}
\item the eigenvalues of the map $Df (E_2^p)$ (see \eqref{JE2_16}) have the form $\pm i\omega$ ($\omega>0$), which is equivalent to $\mathrm{Tr} \,J(E_2^p) = 0$, where $\mathrm{Tr}$ represents the usual trace operator. Indeed,

\begin{eqnarray} \label{HF}
\nonumber && \mathrm{Tr} \,J(E_2^p) = 0 \\
\nonumber && \\
\nonumber \Leftrightarrow&& \dfrac{pm\beta^2-\left(\sigma+g\right)^2}{\beta\left(\sigma+g\right)} = 0\\
\nonumber && \\
\nonumber \Leftrightarrow&& pm\beta^2-\left(\sigma+g\right)^2 = 0 \\
\nonumber && \\
\nonumber \Leftrightarrow && p = \dfrac{\left(\sigma+g\right)^2}{m\beta^2} \\
\nonumber \\
\nonumber \Leftrightarrow&& p = \dfrac{A^2}{m\mathcal{R}_0^2} \\
\nonumber \\
\nonumber \overset{\eqref{pes_eq}}{\Leftrightarrow}&& p = p_H (\mathcal{R}_0).
\end{eqnarray}

\bigskip
\bigskip

\item $\dfrac{\mathrm{d}}{\mathrm{d}\mathcal{R}_0} \left(\operatorname{Re} \lambda_j(\mathcal{R}_0)\right) \big|_{\mathcal{R}_0 \in p_H} \neq 0$, \quad for $j \in \{1,2\}$:

\bigskip

Hence, for $p=p_H(\mathcal{R}_0)$ we get

\begin{eqnarray} \label{realpartalpha}
\nonumber \dfrac{\mathrm{d}}{\mathrm{d}\mathcal{R}_0} \left(\operatorname{Re} \lambda_j(\mathcal{R}_0)\right) \big|_{\mathcal{R}_0 = \mathcal{R}_0^{\star}} &=& \dfrac{\mathrm{d}}{\mathrm{d}\mathcal{R}_0} \left( \dfrac{pm\beta^2}{2\beta\left(\sigma+g\right)} - \dfrac{\left(\sigma+g\right)^2}{2\beta\left(\sigma+g\right)}\right) \Big|_{\mathcal{R}_0 = \mathcal{R}_0^{\star}} \\
\nonumber \\
\nonumber &=& \dfrac{\mathrm{d}}{\mathrm{d}\mathcal{R}_0} \left(\dfrac{\sigma + g}{2\beta} - \dfrac{A}{2\mathcal{R}_0}\right)  \Big|_{\mathcal{R}_0 = \mathcal{R}_0^{\star}} \\
\nonumber \\
\nonumber &=& \dfrac{A}{2{\mathcal{R}_0^{\star}}^2} \neq 0.\\
 \end{eqnarray}
\end{enumerate}

\medskip

Hence, $E_2^p$ undergoes a \emph{subcritical Hopf bifurcation} (as $\mathcal{R}_0$ increases; the periodic solution appears for $\mathcal{R}_0<\mathcal{R}_0^{\star}$ as   illustrated in Figure \ref{HOPF_bifurcation}. Figure \ref{ZH_esquema} provides a better perspective of all bifurcations under consideration.
\end{proof}

\subsection{Heteroclinic cycle} \label{HetCycleSub}
\label{s:Heteroclinic}
Finding explicitly a non-robust \emph{heteroclinic cycle} (in the phase space) to two hyperbolic saddles and its bifurcating curve  is a 
difficult task. 
In this section, we use polynomial interpolation (of rational degree) to find the latter curve. We proceed to explain the method that we have used to compute the map. All steps are described in Appendix \ref{apend1}.

Using \verb|MATLAB_R2018a|, we locate thirteen pairs $({\mathcal{R}_0}_i,p_i)$ in the parameter space $(\mathcal{R}_0,p)$ where one observes a \emph{heteroclinic cycle} associated to the disease-free equilibria $E_0^{p_i}$ and $E_1^{p_i}$ (see Table \ref{het_points}). 
The type of function for the interpolation should be appropriately chosen to fit the curve to our finite set of points~\cite[Part I, Section 1]{Rovenski2010}. 

As suggested by~\cite{Yagasaki2002}, points in the parameter space that correspond to the \emph{heteroclinic cycle} associated to the disease-free equilibria should lie on the graph of: $$y = ax^b + c \quad \text{for}\quad a,b,c \in \mathbb{R}.$$ Using again  \verb|MATLAB_R2018a|, we obtain $a=4.495$, $b=  -2.313$ and $c= - 0.039$, and the graph of $p_{\text{Het.}}\left(\mathcal{R}_0 \right) = \dfrac{4.495}{\mathcal{R}_0^{2.313}} - 0.039$ is plotted in Figure \ref{interp}. 

The {\it heteroclinic bifurcation} organizes the dynamics, stressing a geometric configuration in its unfolding.


\section{Numerics} \label{Num_simul}

In this section we describe the dynamics of the bifurcations of model \eqref{with_vaccine} and we perform numerical simulations associated to parameters in all hyperbolic (open) regions of Figures \ref{ZH_esquema} and \ref{subplots}.

\begin{figure}[h!]
\center
\includegraphics[scale=0.43]{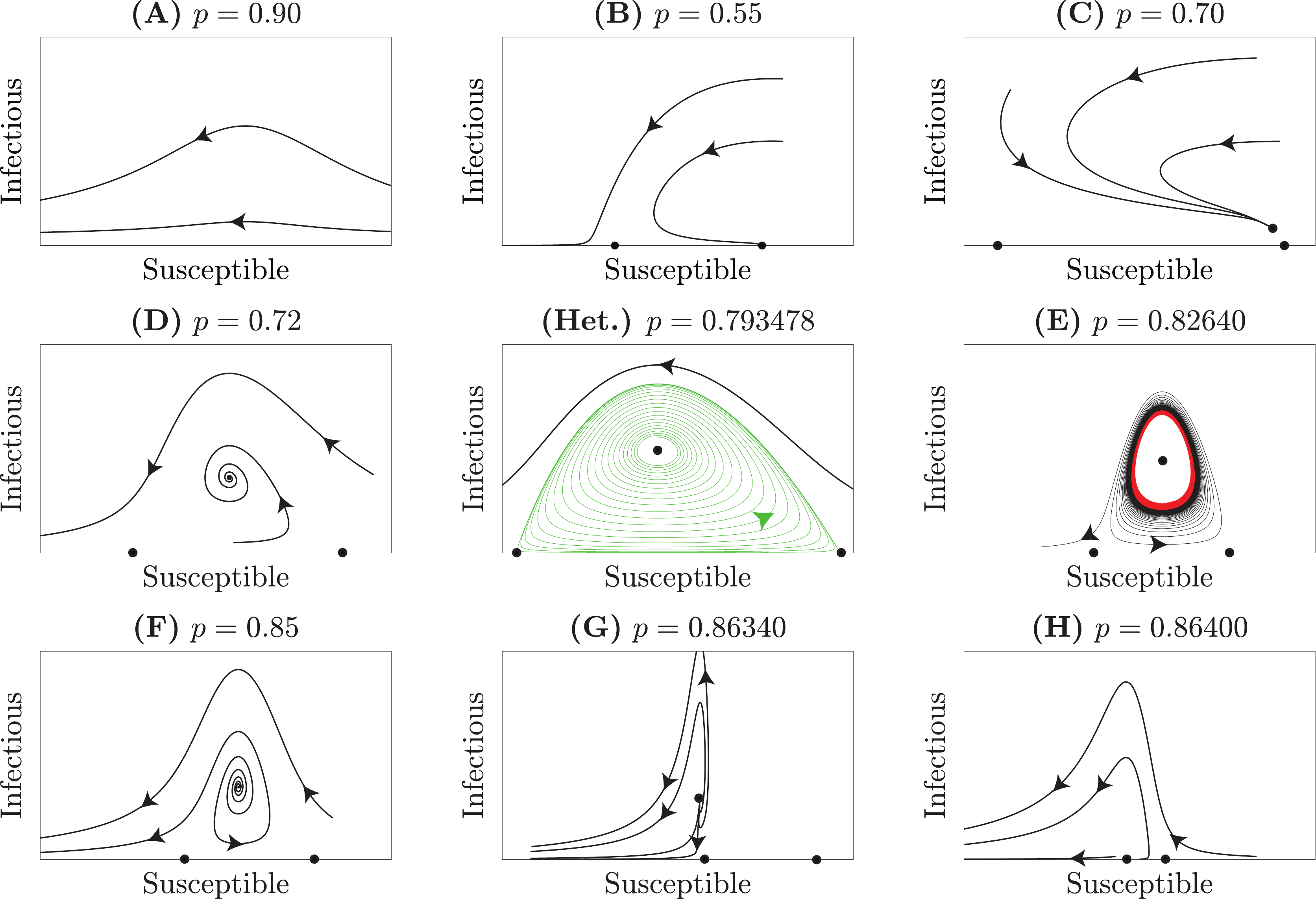}
\caption{Numerical simulations of DZ \emph{bifurcation} diagram with $A = 1.1$, $\beta = 1.3$, $m = \sigma = g = 0.35$ for regions {\bf A}, {\bf D}, {\bf Het.}, {\bf E}, {\bf F}, {\bf G} and {\bf H}. For region {\bf B} we use $A = 1$, $\beta = 1.3$, $m = 0.35$ and $\sigma = g = 0.50$, and for region {\bf C} we use $A = 1.1$, $\beta = 0.91$ and $m = \sigma = g = 0.35$. For all regions we consider $p \in  [0, 1]$. Compare with the theoretical description in Figure \ref{ZH_esquema}. }
\label{subplots}
\end{figure}

\begin{enumerate}
\item[] {\bf A:}   The vaccination coverage is so large that   \emph{Infectious} individuals tend to disappear.  \\
 
\item[] {\bf B:} The flow exhibits two disease-free equilibria, a saddle $E_0^p$ and a sink $E_1^p$.  All trajectories starting in the first quadrant  evolve in  such a way that their $I$ component tend to disappear.  No endemic equilibria exist.\\
 
\medskip
\item[] {\bf C and D:} The endemic equilibrium $E_2^p$ lies in the interior of the first quadrant and it is Lyapunov stable. There is a positive Lebesgue measure set of initial conditions converging to  $E_2^p$, so the disease remains persistently in the population. Initial conditions lying ``below''  $W^s(E_0^p)$ have  $E_2^p$ as $\omega$-limit. \\

\item[] {\bf E:} The unstable \emph{heteroclinic cycle} $\mathcal{H}$ is broken giving rise to an unstable periodic solution $\mathcal{C}$. The region bounded by $\mathcal{C}$ is conducive to the persistence of the disease since all initial conditions tend toward $E_2^p$.
All initial conditions lying in the unbounded region defined by  $\mathcal{C}$ are propitious to the disappearance of the disease.  The period of $\mathcal{C}$   is very large when $(\mathcal{R}_0, p)$ is close to the curve {\bf Het}. 
\medskip
\item[] {\bf F and G:}  The equilibrium 
  $E_2^p$ is   unstable  and all initial conditions are in a region conducive to the disappearance of the disease.
 \medskip
\item[] {\bf H:}  No endemic equilibria exist. The equilibria  $E_0^p$ and $E_1^p$ repel Lebesgue almost all trajectories towards the origin. 
 \end{enumerate}

\section{Concluding remarks} \label{DandC}

In this article, we have performed a bifurcation analysis of a modified SIR model accomodating constant vaccination and logistic growth in the \emph{Susceptible} population, which may be seen as a contribution towards the study of strategies to control infectious diseases.

\subsection{Conclusions}

Model \eqref{modeloSIR} exhibits an endemic \emph{Double-zero} (DZ) \emph{singularity}. To prove the main result,  we have checked that the vector field associated to \eqref{modeloSIR} has a double zero  eigenvalue  (at the endemic equilibrium $E_2^p$) and we have described all associated bifurcation curves passing through the bifurcation point $\left(\mathcal{R}_0^{\star},p^\star\right) = \left(2,\frac{A^2}{4m}\right)$ in the parameter space. In Figure \ref{ZH_esquema}, we have exhibited explicitly  the regions in the space of parameters where the disease persists in the population.   

We have established a threshold for the disease to be extinct or endemic and we have analyzed the existence and asymptotic stability of equilibria.
 Writing the bifurcation curves as a function   $\mathcal{R}_0$ and   $p$ is useful since these variables are key parameters in many epidemiological studies.

The inclusion of a constant vaccination  and a logistic function in  the \emph{Susceptible} population increases the complexity of the dynamics of the classical SIR model. The system may exhibit two disease-free equilibria and none of them corresponds to the equilibrium $(0,0)$. Similar dynamics occurs  in the context of  prey-predator models~\cite{Voorn2007,Olivares2011}. 

Although $Df_{ \left(2,   \frac{A^2}{4m}\right)}$ at $E_2^p$ has a double zero eigenvalue, the associated bifurcation curves  do not correspond to those of the classic \emph{Bogdanov-Takens bifurcation}. It is a variant of the truncated amplitude system associated to the \emph{Hopf-zero} normal form (8.81) of~\cite{Kuznetsov2004}. This implies chaos and suspended horseshoes in the presence of seasonal parameters near $p_H(\mathcal{R}_0)$, as a consequence of the theory developed in~\cite{CarvalhoRodrigues2022, RodriguesBykov2022}.

In the presence of the endemic equilibrium (in the first quadrant of $(S,I)$), we have detected two different regimes for $\mathcal{R}_0 \geq 1$:
\begin{itemize}
\item $1<\mathcal{R}_0  <2$: the increase of the vaccination rate plays an essencial role to decrease the number of {\it Infectious};
\item $2\leq \mathcal{R}_0 $: although vaccination policies are beneficial, the control of the endemic disease is mostly due to the saturation of the class of {\it Susceptible} individuals. 
\end{itemize}


With constant vaccination, the {\it transcritical} and the {\it Hopf bifurcation} curves induce a change in the stability of the endemic equilibrium point. At the epidemiological level, this means that these curves determine if the disease either disappears or remains. {\it The saddle-node bifurcation} curves indicate the moment where disease-free equilibria appear.

Although the model under analysis has limitations in terms of biologic validity, the existence of a DZ {\it bifurcation} organizes the dynamics and stresses the role played by the vaccination; the model agrees well with the empirical beliefs.

\subsection{Future work}
  As already referred, our work generalizes that of Shulgin \emph{et al.}~\cite{Shulgin1998} (with constant vaccination) who modelled individual growth just with birth rate. 
  
    The strategy of constant vaccination  may not be the most effective if the proportion of successfully vaccinated newborns is low. It may be necessary to adopt more effective strategies such as the {\it pulse vaccination}. This technique  aims to find an optimal period between ``shots'' of the vaccine, allowing a proportion of infected individuals lower than a given threshold. It  should be possible to apply this technique to the regions where the disease persists 
({\bf C},  {\bf D} and {\bf E} of Figure \ref{ZH_esquema}).  We would like to derive the period of the  vaccination ``shot'' that (efficiently)  allows to control the number of infected individuals.

{\it Pulse vaccination} can have different effects on the dynamics of infectious diseases, depending on our starting model. Our next goal is to modify  the first component of \eqref{modeloSIR}  in the following way: 

$$ \dot{S} =  S(A-S) - \beta IS - p \displaystyle \sum_{n=0}^{+\infty} S(nT^-) \delta(t-nT)$$
where $n\in \NN$:
\begin{itemize}
\item  $T$ is the period between two consecutive ``shots'' (of vaccination);
\item  $\dpt S(nT^-)=\lim_{t\rightarrow n T^-} S(t)$ is the number of {\it Susceptible} individuals at the instant immediately before being vaccinated and  $\delta $ is the usual Dirac function.\\
\end{itemize}
The analysis of an impulsive differential equation involving {\it pulse vaccination} is an ongoing work.  
\bigskip
\bigskip

\paragraph{\Large{{\bf Declarations}}} \textcolor{white}{.}

\medskip

\noindent {\bf Conflict of interest.} The authors declare that they have no conflict of interest.

\medskip

\noindent Jo\~ao Carvalho and Alexandre Rodrigues have  equally contributed to this work.

\bigskip
 
\appendix
\par
\section{MATLAB R2018a: code and procedure}

We set the code  of \verb|MATLAB_R2018a|   to obtain the interpolation function that best fits the $({\mathcal{R}_0}_i, p_i)$ points in the parameter space $(\mathcal{R}_0, p)$.
\label{apend1}
\medskip

\verb|f=fit(x,y,`power2')|

\verb|plot(f, x, y)|

\verb|xlim([2 3.8])|

\verb|ylim([0.1 0.75])|

\verb|ax = gca;|

\verb|ax.FontSize = 30;|

\verb|xlabel(`$\mathcal{R}_0$',`Interpreter',`latex')|

\verb|ylabel(`$p$',`Interpreter',`latex')|

\verb|leg1 = legend(`(${\mathcal{R}_0}_i,p_i$) points',`Het. cycle function')|

\verb|lgd.FontSize = 20;|

\verb|legend boxoff|

\verb|set(leg1,`Interpreter',`latex');|

\smallskip

\noindent where \verb|x| and \verb|y| are the vectors of ${\mathcal{R}_0}_i$ and $p_i$ values, respectively. The interpolation function given by the \verb|MATLAB_R2018a|   has a correlation coefficient equal to $1$ (precision: $10^{-5}$; confidence bound: 95\%), as we can confirm in Figure \ref{fit_R}.

\begin{table}[h!]
\centering
\begin{tabular}{ |ccc|  }
 \hline
$i$ & ${\mathcal{R}_0}_i$ & $p_i$  \\  [0.29ex]
 \hline  
1 & \verb|2.0725| & \verb|0.793486|  \\ 
2 & \verb|2.2000| & \verb|0.686625|  \\ 
3 & \verb|2.2698| & \verb|0.636156|  \\ 
4 & \verb|2.4237| & \verb|0.541135|  \\ 
5 & \verb|2.6000| & \verb|0.453994|  \\ 
6 & \verb|2.6981| & \verb|0.413374|  \\ 
7 & \verb|2.8039| & \verb|0.374719|  \\ 
8 & \verb|2.9184| & \verb|0.338027|  \\ 
9 & \verb|3.0426| & \verb|0.303294|  \\ 
10 & \verb|3.1778| & \verb|0.270517|  \\ 
11 & \verb|3.3256| & \verb|0.239692|  \\ 
12 & \verb|3.4878| & \verb|0.210816|  \\ 
13 & \verb|3.6667| & \verb|0.183883|  \\
 \hline
\end{tabular}
\medskip
\cprotect \caption{13 interpolating points of the \emph{heteroclinic cycle} function in the space of parameters $(\mathcal{R}_0,p)$. These points were obtained in \verb|MATLAB_R2018a|.}
\label{het_points}
\end{table}

\begin{figure}[h!]
\center
\includegraphics[scale=0.42]{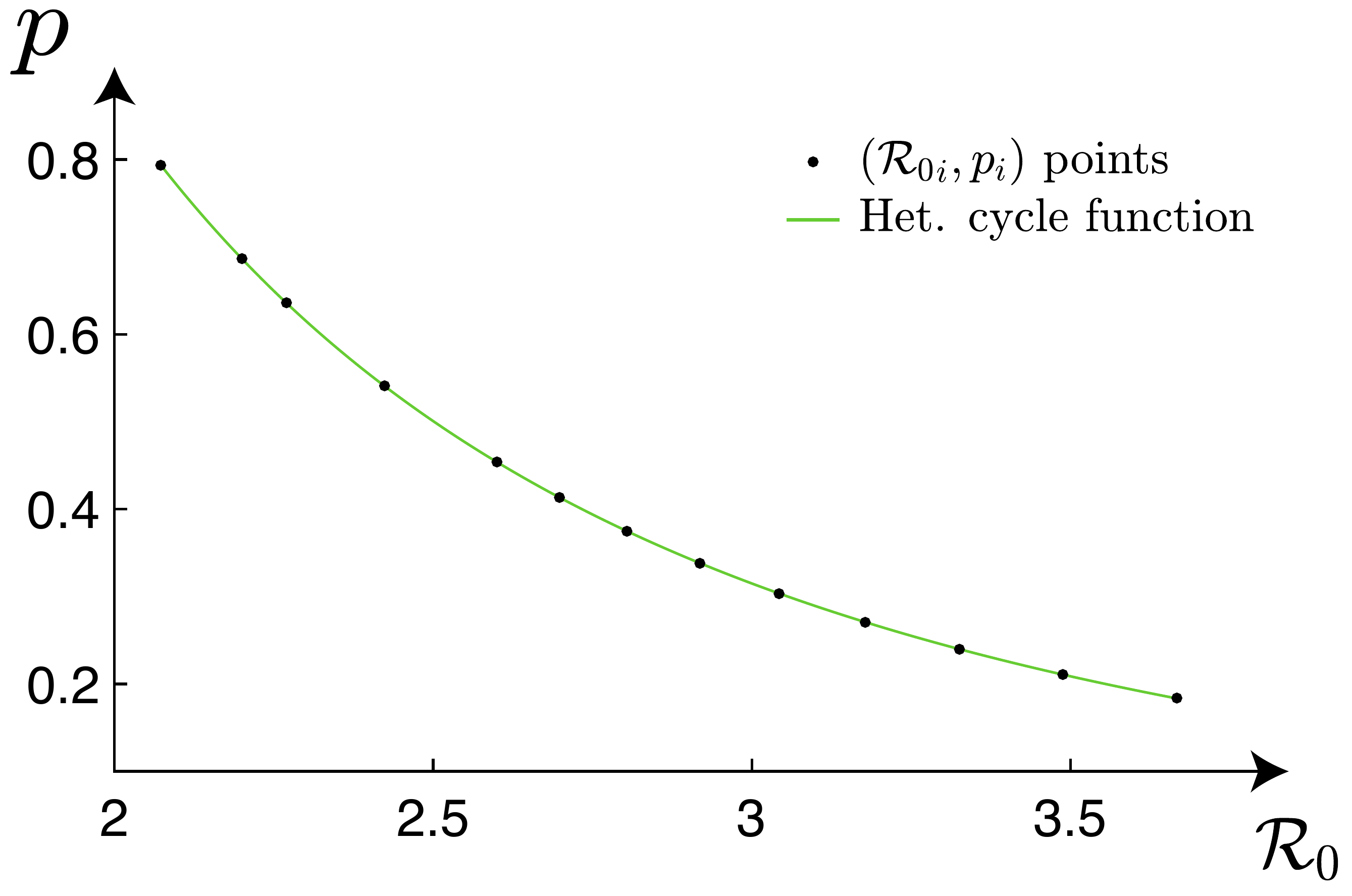}
\caption{Function obtained by interpolation of 13 points (see Table \ref{het_points}) for which it is possible to observe the \emph{heteroclinic cycle} in phase space $(S,I)$. The function obtained by interpolation for which we can find the \emph{heteroclinic cycle} in the space of parameters $(\mathcal{R}_0,p)$ is $p\left(\mathcal{R}_0 \right) \approx 4.495 \mathcal{R}_0^{-2.313} - 0.039$.}
\label{interp}
\end{figure}

\begin{figure}[h!]
\center
\includegraphics[scale=0.75]{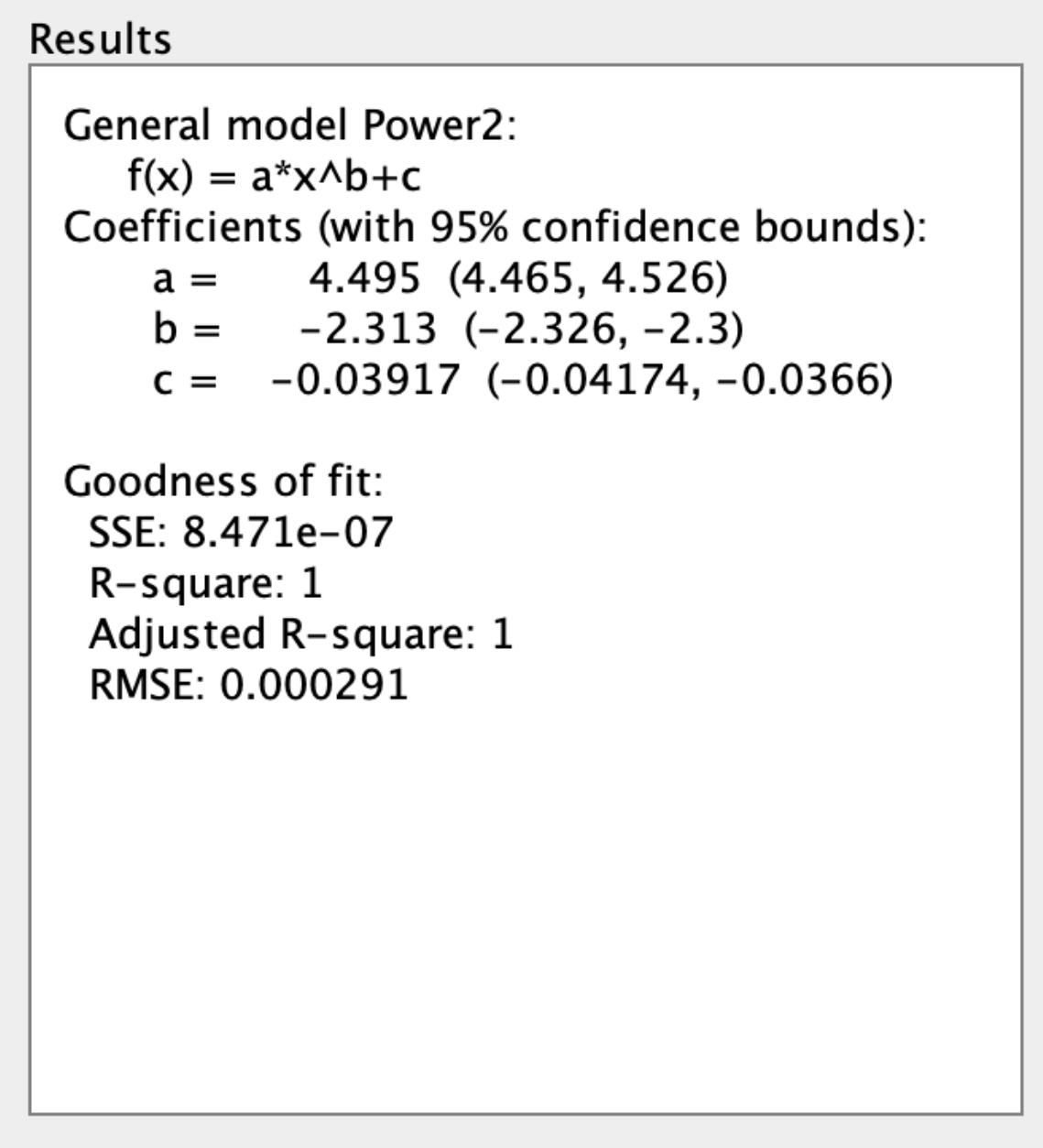}
\cprotect \caption{\verb|MATLAB_R2018a| output for the two-term power series model used in $({\mathcal{R}_0}_i,p_i)$ points interpolation.}
\label{fit_R}
\end{figure}
\setlength{\tabcolsep}{12pt}
\renewcommand{\arraystretch}{1.35}


\newpage 
\begin{thebibliography}{777}
\bibliographystyle{unsrt}

\bibitem{Brauer2019} 
Brauer, F., Castillo-Chavez, C., Feng, Z.: Mathematical Models in Epidemiology. Texts in Applied Mathematics vol. 69. Springer New York, NY (2019) \textcolor{blue}{https://doi.org/10.1007/978-1-4939-9828-9}

\bibitem{Hethcote2000}  
Hethcote, H.W.: The mathematics of infectious diseases. SIAM Rev. 42, 599--653 (2000). \textcolor{blue}{https://doi.org/10.1137/S0036144500371907}

\bibitem{Bonyah2020} 
Bonyah, E., Al Basir, F., Ray, S.: Hopf bifurcation in a mathematical model of tuberculosis with delay, in: P. Manchanda, R. Lozi, A. Siddiqi (Eds.), Mathematical Modelling, Optimization, Analytic and Numerical Solutions, in: Industrial and Applied Mathematics. Springer, Singapore, 2020, pp. 301--311

\bibitem{Rajagopal2020} 
Rajagopal, K., Hasanzadeh, N., Parastesh, F., Hamarash, I.I., Jafari, S., Hussain, I.: A fractional-order model for the novel coronavirus (COVID-19) outbreak. Nonlinear Dynam. 101, 711--718 (2020). \textcolor{blue}{https://doi.org/10.1007/s11071-020-05757-6}

\bibitem{Cobey2020} 
Cobey, S.: Modeling infectious disease dynamics. Science 368, 713--714 (2020). \textcolor{blue}{http://doi.org/10.1126/science.abb5659}

\bibitem{Kermack1932} 
Kermack, W.O., McKendrick, A.G.: Contributions to the mathematical theory of epidemics. II. -- The problem of endemicity, Proc. R. Soc. Lond. 138, 55--83 (1932). \textcolor{blue}{https://doi.org/10.1098/rspa.1932.0171}

\bibitem{Dietz1976} 
Dietz, K.: The incidence of infectious diseases under the influence of seasonal fluctuations. In: Mathematical models in medicine. Springer Berlin Heidelberg, pp 1--15 (1976). \textcolor{blue}{https://doi.org/10.1007/978-3-642-93048-5\_1}

\bibitem{MilnerPugliese1999} 
Milner, F.A., Pugliese, A.: Periodic solutions: a robust numerical method for an S-I-R model of epidemics. J. Math. Biol. 39, 471--492 (1999). \textcolor{blue}{https://doi.org/10.1007/s002850050175}

\bibitem{CarvalhoRodrigues2022} 
Maur\'icio de Carvalho, J.P.S., Rodrigues, A.A.P.: Strange attractors in a dynamical system inspired by a seasonally forced SIR model. Physica D. 434, 12 pages (2022). \textcolor{blue}{https://doi.org/10.1016/j.physd.2022.133268}

\bibitem{Barrientos2017} 
Barrientos, P.G., Rodr\'iguez, J.A., Ruiz-Herrera, A.: Chaotic dynamics in the seasonally forced SIR epidemic model. J. Math. Biol. 75, 1655--1668 (2017). \textcolor{blue}{https://doi.org/10.1007/s00285-017-1130-9}

\bibitem{CarvalhoPinto2021} 
Maur\'icio de Carvalho, J.P.S., Moreira-Pinto, B.: A fractional-order model for CoViD-19 dynamics with reinfection and the importance of quarantine. Chaos Solitons Fractals. 151, 7 pages (2021). \textcolor{blue}{https://doi.org/10.1016/j.chaos.2021.111275}

\bibitem{Onofrio2022} 
d'Onofrio, A., Duarte, J., Janu\'ario, C., Martins, N.: A SIR forced model with interplays with the external world and periodic internal contact interplays. Phys. Lett. A. 454, 9 pages (2022). \textcolor{blue}{https://doi.org/10.1016/j.physleta.2022.128498}

\bibitem{Plotkin2005} 
Plotkin, S.A.: Vaccines: past, present and future. Nat. Med. 11, S5--S11 (2005). \textcolor{blue}{https://doi.org/10.1038/nm1209}

\bibitem{Plotkin2011} 
Plotkin, S.A., Plotkin, S.L.: The development of vaccines: how the past led to the future. Nat. Rev. Microbiol. 9, 889--893 (2011). \textcolor{blue}{https://doi.org/10.1038/nrmicro2668}

\bibitem{Makinde2007} 
Makinde, O.D.: Adomian decomposition approach to a SIR epidemic model with constant vaccination strategy. Appl. Math. Comput. 184, 842--848 (2007). \textcolor{blue}{https://doi.org/10.1016/j.amc.2006.06.074}

\bibitem{SahaGhosh2022} 
Saha, P., Ghosh, U.: Complex dynamics and control analysis of an epidemic model with non-monotone incidence and saturated treatment. Int. J. Dyn. Control. 2022, 23 pages (2022). \textcolor{blue}{https://doi.org/10.1007/s40435-022-00969-7}

\bibitem{Ghosh2019} 
Ghosh, J.K., Ghosh, U., Biswas, M.H.A., Sarkar, S.: Qualitative Analysis and Optimal Control Strategy of an SIR Model with Saturated Incidence and Treatment. Differ. Equ. Dyn. Syst. 2019, 15 pages (2019). \textcolor{blue}{https://doi.org/10.1007/s12591-019-00486-8}

\bibitem{Shulgin1998} 
Shulgin, B., Stone, L., Agur, Z.: Pulse Vaccination Strategy in the SIR Epidemic Model. Bull. Math. Biol. 60, 1123--1148 (1998). \textcolor{blue}{https://doi.org/10.1016/S0092-8240(98)90005-2}

\bibitem{Elazzouzi2019} 
Elazzouzi, A., Alaoui, A.L., Tilioua, M., Tridane, A.: Global stability analysis for a generalized delayed SIR model with vaccination and treatment. Adv. Differ. Equ. 2019, 19 pages (2019). \textcolor{blue}{https://doi.org/10.1186/s13662-019-2447-z}

\bibitem{Stone2000} 
Stone, L., Shulgin, B., Agur, Z.: Theoretical Examination of the Pulse Vaccination Policy in the SIR Epidemic Model. Math. Comput. Model. 31, 207--215 (2000). \textcolor{blue}{https://doi.org/10.1016/S0895-7177(00)00040-6}

\bibitem{Algaba2022} 
Algaba, A., Dom\'inguez-Moreno, M.C., Merino, M., Rodr\'iguez-Luis, A.J.: Double-zero degeneracy and heteroclinic cycles in a perturbation of the Lorenz system. Commun. Nonlinear Sci. Numer. Simul. 111, 23 pages (2022). \textcolor{blue}{https://doi.org/10.1016/j.cnsns.2022.106482}

\bibitem{Rodrigues2021} 
Rodrigues, A.A.P.: Dissecting a Resonance Wedge on Heteroclinic Bifurcations. J. Stat. Phys. 184, 32 pages (2021). \textcolor{blue}{https://doi.org/10.1007/s10955-021-02811-4}

\bibitem{Yagasaki2002} 
Yagasaki, K.: Melnikov's Method and Codimension-Two Bifurcations in Forced Oscillations. J. Differ. Equ. 185, 1--24 (2002). \textcolor{blue}{https://doi.org/10.1006/jdeq.2002.4177}

\bibitem{Duarte2019} 
Duarte, J., Janu\'ario, C., Martins, N., Rogovchenko, S., Rogovchenko, Y.: Chaos analysis and explicit series solutions to the seasonally forced SIR epidemic model. J. Math. Biol. 78, 2235--2258 (2019). \textcolor{blue}{https://doi.org/10.1007/s00285-019-01342-7}

\bibitem{Kuznetsov2004} 
Kuznetsov, Y.A.: Elements of Applied Bifurcation Theory. Applied Mathematical Sciences vol 112. Springer, New York, NY (2004). \textcolor{blue}{https://doi.org/10.1007/978-1-4757-3978-7}

\bibitem{Jin2007} 
Jin, Y., Wang, W., Xiao, S.: An SIRS model with a nonlinear incidence rate. Chaos Solitons Fractals. 34, 1482--1497 (2007). \textcolor{blue}{https://doi.org/10.1016/j.chaos.2006.04.022}

\bibitem{Zhang2023} 
Zhang, J., Qiao, Y.: Bifurcation analysis of an SIR model considering hospital resources and vaccination. Math. Comput. Simul. 208, 157--185 (2023). \textcolor{blue}{https://doi.org/10.1016/j.matcom.2023.01.023}

\bibitem{ShanZhu2014} 
Shan, C., Zhu, H.: Bifurcations and complex dynamics of an SIR model with the impact of the number of hospital beds. J. Differ. Equ. 257, 1662--1688 (2014). \textcolor{blue}{https://doi.org/10.1016/j.jde.2014.05.030}

\bibitem{AlexanderMoghadas2005} 
Alexander, M.E., Moghadas, S.M.: Bifurcation analysis of an SIRS epidemic model with generalized incidence. SIAM J. Appl. Math. 65, 1794--1816 (2005). \textcolor{blue}{https://www.jstor.org/stable/4096153}

\bibitem{Pan2022} 
Pan, Q., Huang, J., Wang, H.: An SIRS model with nonmonotone incidence and saturated treatment in a changing environment. J. Math. Biol. 85, 39 pages (2022). \textcolor{blue}{https://doi.org/10.1007/s00285-022-01787-3}

\bibitem{LiTeng2018} 
Li, J., Teng, Z.: Bifurcations of an SIRS model with generalized non-monotone incidence rate. Adv. Differ. Equ. 2018, 21 pages (2018). \textcolor{blue}{https://doi.org/10.1186/s13662-018-1675-y}

\bibitem{Misra2022} 
Misra, A.K., Maurya, J., Sajid, M.: Modeling the effect of time delay in the increment of number of hospital beds to control an infectious disease. Math. Biosci. Eng. 19, 11628--11656 (2022). \textcolor{blue}{https://doi.org/10.3934/mbe.2022541}


\bibitem{Lu2019} 
Lu, M., Huang, J., Ruan, S., Yu, P.: Bifurcation analysis of an SIRS epidemic model with a generalized nonmonotone and saturated incidence rate. J. Differ. Equ. 267, 1859--1898 (2019). \textcolor{blue}{https://doi.org/10.1016/j.jde.2019.03.005}


\bibitem{LiTeng2017} 
Li, J., Teng, Z., Wang, G., Zhang, L., Hu, C.: Stability and bifurcation analysis of an SIR epidemic model with logistic growth and saturated treatment. Chaos Solitons Fractals 99, 63--71 (2017). \textcolor{blue}{https://doi.org/10.1016/j.chaos.2017.03.047}

\bibitem{StoneShulgin2000} 
Stone, L., Shulgin, B.: Theoretical Examination of the Pulse Vaccination Policy in the SIR Epidemic Model. Math. Comput. Model. 31, 207-215 (2000). \textcolor{blue}{https://doi.org/10.1016/S0895-7177(00)00040-6}

\bibitem{LuChiChen2002} 
Lu, Z., Chi, X., Chen, L.: The Effect of Constant and Pulse Vaccination on SIR Epidemic Model with Horizontal and Vertical Transmission. Math. Comput. Model. 36, 1039--1057 (2002). \textcolor{blue}{https://doi.org/10.1016/S0895-7177(02)00257-1}

\bibitem{ZhangChen1999} 
Zhang, X.A., Chen, L.: The Periodic Solution of a Class of Epidemic Models. Comput. Math. with Appl. 38, 61--71 (1999). \textcolor{blue}{https://doi.org/10.1016/S0898-1221(99)00206-0}

\bibitem{Jones2007} 
Jones, J.H.: Notes on $\mathcal{R}_0$. California: Department of Anthropological Sciences 323, 19 pages (2007)

\bibitem{ParkBolker2020} 
Park, S.W., Bolker, B.M.: A Note on Observation Processes in Epidemic Models. Bull. Math. Biol. 82, 8 pages (2020). \textcolor{blue}{https://doi.org/10.1007/s11538-020-00713-2}

\bibitem{Li2011} 
Li, J., Blakeley, D., Smith, R.J.: The failure of $\mathcal{R}_0$ (2011) Comput. Math. Methods Med. 2011, 17 pages (2011). \textcolor{blue}{https://doi.org/10.1155/2011/527610}

\bibitem{WangYoung2003} 
Wang, Q., Young, L.S.: Strange Attractors in Periodically-Kicked Limit Cycles and Hopf Bifurcations. Commun. Math. Phys. 240, 509--529 (2003). \textcolor{blue}{https://doi.org/10.1007/s00220-003-0902-9}

\bibitem{GuckenheimerHolmes1983} 
Guckenheimer, J., Holmes, P.J.: Nonlinear Oscillations, Dynamical Systems, and Bifurcations of Vector Fields. Applied Mathematical Sciences vol. 42. Springer New York, NY (1983) \textcolor{blue}{https://doi.org/10.1007/978-1-4612-1140-2}

\bibitem{Rovenski2010} 
Rovenski, V.: Modeling of Curves and Surfaces with MATLAB. Springer Undergraduate Texts in Mathematics and Technology (SUMAT). Springer New York, NY (2010)

\bibitem{Voorn2007} 
van Voorn, G.A.K., Hemerik, L., Boer, M.P., Kooi, B.W.: Heteroclinic orbits indicate overexploitation in predator-prey systems with a strong Allee effect. Math. Biosci. 209, 451--469 (2007). \textcolor{blue}{https://doi.org/10.1016/j.mbs.2007.02.006}

\bibitem{Olivares2011} 
Gonz\'alez-Olivares, E., Gonz\'alez-Ya\~nez, B., Lorca, J.M., Rojas-Palma, A., Flores, J.D.: Consequences of double Allee effect on the number of limit cycles in a predator-prey model. Comput. Math. with Appl. 62, 3449--3463 (2011). \textcolor{blue}{https://doi.org/10.1016/j.camwa.2011.08.061}

\bibitem{RodriguesBykov2022} 
Rodrigues, A.A.P.: Unfolding a Bykov Attractor: From an Attracting Torus to Strange Attractors.  J. Dyn. Differ. Equ. 34, 1643--1677 (2022). \textcolor{blue}{https://doi.org/10.1007/s10884-020-09858-z}

\end{thebibliography}
\end{document}